\documentclass[12pt]{amsart}
\usepackage{amsmath}
\usepackage[all]{xypic}
\usepackage{graphicx}
\usepackage{amssymb,amsthm}
\usepackage{hyperref}
\usepackage{esint}
\usepackage{epstopdf}
\usepackage{color}
\usepackage{url}
\usepackage{enumerate}
\usepackage[top=1in, left=1in, right=1in, bottom=1in]{geometry}

\newtheorem{theorem}{Theorem}[section]
\newtheorem{lemma}[theorem]{Lemma}
\newtheorem{proposition}[theorem]{Proposition}
\newtheorem{corollary}[theorem]{Corollary}

\newtheorem{remark}[theorem]{Remark}
\newtheorem{definition}{Definition}[section]

\def\R{{\mathbb R}}

\def\cA{{\mathcal A}}
\def\cB{{\mathcal B}}

\def\cH{{\mathcal H}}

\def\cL{{\mathcal L}}

\def\cQ{{\mathcal Q}}

\def\a{\alpha}
\def\b{\beta}
\def\e{\varepsilon}
\def\d{\delta}
\def\D{\triangle}

\def\k{\kappa}
\def\l{\lambda}
\def\L{\Lambda}
\def\m{\mu}
\def\n{\nabla}
\def\p{\partial}
\def\r{\rho}
\def\s{\sigma}
\def\t{\tau}
\def\U{\Upsilon}
\def\w{\omega}
\def\W{\Omega}
\def\g{\gamma}
\def\z{\zeta}

\def\1{\left(}
\def\2{\right)}
\def\3{\left\{}
\def\4{\right\}}
\def\8{\infty}
\def\sm{\setminus}
\def\ss{\subseteq}

\newcommand{\mres}{\mathbin{\vrule height 1.6ex depth 0pt width
0.13ex\vrule height 0.13ex depth 0pt width 1.3ex}}

\DeclareMathOperator*{\dvg}{div}

\DeclareMathOperator*{\osc}{osc}

\begin{document}

\title{Regularity for Shape Optimizers: The Nondegenerate Case}

\author{Dennis Kriventsov}
\address[Dennis Kriventsov]{Courant Institute of Mathematical Sciences, New York University, New York}
\email{dennisk@cims.nyu.edu}  

\author{Fanghua Lin}
\address[Fanghua Lin]{Courant Institute of Mathematical Sciences, New York University, New York}
\email{linf@cims.nyu.edu}

\date{June 14, 2017}

\begin{abstract}
	We consider minimizers of
	\[
	F(\l_1(\W),\ldots,\l_N(\W)) + |\W|,
	\]
	where $F$ is a function strictly increasing in each parameter, and $\l_k(\W)$ is the $k$-th Dirichlet eigenvalue of $\W$. Our main result is that the reduced boundary of the minimizer is composed of $C^{1,\a}$ graphs, and exhausts the topological boundary except for a set of Hausdorff dimension at most $n-3$. We also obtain a new regularity result for vector-valued Bernoulli type free boundary problems. 
\end{abstract}

\maketitle

\section{Introduction}\label{sec:intro}

We consider the problem of minimizing a functional of the form 
\begin{equation}\label{eq:one}
F\1\l_1(\W),\ldots,\l_N(\W)\2+|\W|,
\end{equation}
where $\W\ss \R^n$ is an open set and $\l_j(\W)$ is the $j$-th Dirichlet eigenvalue of the Laplacian on $\W$. Here $F=F(\xi_1,\ldots,\xi_N)$ is a $C^1$ function which is increasing in the parameters, and has $\p_{\xi_j}F\geq \l>0$. This second property plays a key role in the arguments below, and we believe that some of the conclusions presented may fail were it to be omitted. Provided the function $F$ is homogeneous, an equivalent formulation of this problem is to minimize
\begin{equation}\label{eq:onestar}
F\1\l_1(\W),\ldots,\l_N(\W)\2
\end{equation}
over all sets of a fixed measure; minimizers of these two functionals are dilates of each other. 

For example, our results apply to the linear combinations of the eigenfunctions
\[
\sum_{k=1}^N \m_k \l_k(\W),
\]
to functionals like
\[
\sum_{k=1}^N \l_k^p
\]
or to maximizers of
\[
\sum_{k=1}^N \frac{1}{\l_k^p}.
\]

Functionals of this type figure in many classical results and inequalities related to the spectrum of the Laplacian, and the shape of minimizers is usually understood imperfectly at best. When $N=1$, it is well-known that the sole minimizer is a ball; however, even in the case of $N=2$ there are open questions about the minimizer's shape. We refer the reader to \cite{H} for further discussion, references, and many open problems.

Until relatively recently, general properties of minimizers were not understood. A major breakthrough was achieved in \cite{DB}, where such functionals were shown to admit minimizers in the class of \emph{quasiopen} sets--sets which are positivity sets of functions in the Sobolev space $H^1$--when restricted to some ball. This was recently improved in the works \cite{Bucur,MP} (which discussed the especially interesting case $F(\W) = \l_N$, but apply also under our assumptions), where it was shown that minimizers may be found without the restriction to a compact region, that they are open sets, and that they have finite perimeter. A further result, and one which we will use extensively, is that for a functional under our assumptions, any eigenfunction corresponding to $\l_1,\ldots,\l_N$ is Lipschitz continuous \cite{BMPV}.

In this paper, we consider the regularity of the boundary $\p \W$. Our main result is the following:

\begin{theorem}\label{thm:shapes} Let $\W$ be a minimizer of \eqref{eq:one}, and let $\p^* \W$ be its reduced boundary. Then $\p^*\W$ is a relatively open set locally composed of graphs of $C^{1,\a}$ functions. Furthermore, the Hausdorff dimension of $\p \W\sm \p^* \W$ is at most $n-3$; in particular, if $n=2$, then $\p \W $ is a union of finitely many closed $C^{1,\a}$ curves. 
\end{theorem}

\begin{remark} As mentioned earlier, a minimizer of \eqref{eq:onestar} will also be a minimizer of a functional of the form \eqref{eq:one} provided $F$ is homogeneous. If $F$ is not homogeneous, then a minimizer of \eqref{eq:onestar} will minimize a functional of the form
	\[
	F\1\l_1(\W),\ldots,\l_N(\W),|\W|\2,
	\]
	satisfying the same assumption of $\p_{\xi_j}F$ bounded above and below, but now also in the last parameter. Our method can be applied to this kind of functional as well with superficial changes, but we will only treat the simpler form \eqref{eq:one} below to avoid complicating the exposition.
\end{remark}

\begin{remark} The reduced boundary $\p^*\W$ in Theorem \ref{thm:shapes} will actually be (locally) analytic; the same is true in Theorem \ref{thm:fb} below. This is explained in the Appendix, see Theorem \ref{thm:higherreg}. If the first $N$ eigenvalues of $\W$ are simple, a better estimate is available on the singular set, see Theorem \ref{thm:sing}.
\end{remark}

The method of proof is based on the observation that, under favorable circumstances, performing domain variations on minimizers of $F$ leads to a stationarity condition of the form
\begin{equation}\label{eq:two}
\sum_{k=1}^N \p_{\l_k}F (u_k)_\nu^2 = 1.
\end{equation}
Here $u_k$ is the $k-$th eigenfunction, and $\nu$ is an outward unit normal. This may be interpreted as a kind of vector-valued Bernoulli-type free boundary problem. Scalar versions of this free boundary problem were studied very successfully in \cite{AC}, and then in greater generality and different techniques in \cite{C1,C2,C3}. A different approach was introduced in \cite{DeSilva}, which is the one we adapt to the vectorial setting.

There are three main difficulties with applying known free boundary results to this setting, however. First, while under special assumptions (such as if the first $N$ eigenvalues of the minimizer turn out to be simple) it is indeed possible to obtain a weak version of \eqref{eq:two}, in general $F$ may fail to be differentiable with respect to domain variations entirely. This happens whenever $\l_k = \l_{k+1}$ for some eigenvalue, and also $\p_{\l_k}F < \p_{\l_{k+1}} F$ when evaluated on the spectrum of the optimal set. In addressing this, we were inspired by the recent paper \cite{RTT}, where the authors studied optimal partition problems for functionals featuring higher eigenvalues. Encountering a similar issue, they observed that it is always possible to approximate $F$ by functionals which \emph{do} admit a domain variation formula. Then the condition \eqref{eq:two} is passed to the limit, giving
\begin{equation*}
\sum_{k=1}^N \xi_k (u_k)_\nu^2 = 1.
\end{equation*}
The relation between $\xi_k$ and the partial derivative of $F$ is lost, but we still have a free boundary condition to work with.

The other two difficulties are free-boundary related: this problem is vectorial, and the functions $u_k$ may change sign. We adapt the method of De Silva to handle this situation. The argument still proceeds by iteratively zooming in, and trying to decrease the distance between the eigenfunctions $\{u_k\}$ and their tangent object (which is a collection of half-plane functions $\{\a_k (x_n)^-\}$). We distinguish between those eigenfunctions which look like they may have vanishing normal derivative $u_\nu$ at the point we are targeting (this means that point is at the end of a nodal curve) and those which look like they have linear growth. An eigenfunction might transition from the first category to the second, but not the other way around; such transitions are then handled in a separate compactness argument. The rest of the argument is based on the Harnack inequality and linearizion of De Silva, although they require  using some specialized families of barriers. 

In particular, our main theorem also implies the following result for free boundaries:
\begin{theorem}\label{thm:fb} Let $\{ u_k\}$ be a family of harmonic functions on an open set $\W\ss B_1$, which satisfy
	\[
	\sum_{k=1}^N \xi_k (u_k)_\nu^2 = 1
	\]
	in the \emph{viscosity sense} on $\p \W$. Assume that the $u_k$ are all Lipschitz, and satisfy, for any $x\in \p \W$,
	\[
	\sup_{B_r(x),k}|u_k| \geq c r. 
	\]
	If $\p \W\ss \{|x_n|\leq \e\}$ for $\e$ small enough, and also $(0,1/2)\notin \W$, $(0,-1/2)\in \W$, then $\p \W \cap B_{1/2}$ is a $C^{1,\a}$ graph over $\{x_n = 0\}$.
\end{theorem}
We do not assume that any of the $u_k$ have a sign.  The definition of viscosity solution is given in Section \ref{sec:propE}. 

The above theorem assumes a specific geometric configuration, where $\p \W$ is trapped between two parallel hyperplanes a distance $\e$ apart, and moreover the set $\W$ lies to one side of the hyperplanes (and its complement to the other side). That $\p \W$ lies in this thin region is well-known to be a necessary condition for this type of problem (there are straightforward counterexamples when $n=3$, $N=1$, and $u_1$ is nonnegative, see \cite{AC} for discussion). The extra assumption that $\W$ lies to one side of the hyperplanes is also necessary here, as it is not difficult to construct cusp-like solutions with $N=1$ where $\W$ occupies both sides of the hyperplanes, $\p \W\cup \{x\in \W: u_1=0\}$ is the union of two smooth hypersurfaces (each a graph over $\{x_n=0\}$), and $u_1$ is positive to one side of them, negative to the other side, and zero in between. In this case, $\p \W$ could be taken to be the region where the hypersurfaces do not touch, which may have many cusp-like singularities. We do not address the regularity of such a scenario in this work, and these cusps will not occur in the situation of Theorem \ref{thm:shapes}.

As this paper was in the final stages of preparation, we learned of two very recent preprints, \cite{CSY} and \cite{MTV}. In \cite{CSY}, the authors prove a theorem like \ref{thm:fb}, but under the assumption that all of the $u_k$ are positive. In \cite{MTV}, a different group of authors use an argument of the same flavor to prove a theorem like \ref{thm:shapes}, at least for those functionals which are differentiable with respect to domain variations (they focus on $\sum_{k=1}^N \l_k$); they also can prove a version of Theorem \ref{thm:fb} under the assumption that at least one of the $u_k$ is positive. Both of these proofs are very elegant and simple, using the boundary Harnack inequality to reduce the vector problem to an already solved scalar one. As a lone redeeming feature of our longer and more technical argument, we suggest that our method may be useful when considering functionals not involving the first eigenvalue. As a final remark, \cite{MTV} includes much simpler proofs of some of the auxiliary lemmas we prove here (notably, our Lemma \ref{lem:stronglb}). We have elected to retain our arguments, but recommend the reader to look at \cite{MTV} as well.

To aid the reader in comparing these different results, and possible future extensions, we suggest the following schematic about assumptions which may be reasonably placed on $F$ (arrows indicate additional assumptions being made, so the hypotheses get stronger as one moves lower down; double arrows indicate equivalent hypotheses; $\W$ is referring to a minimizer):
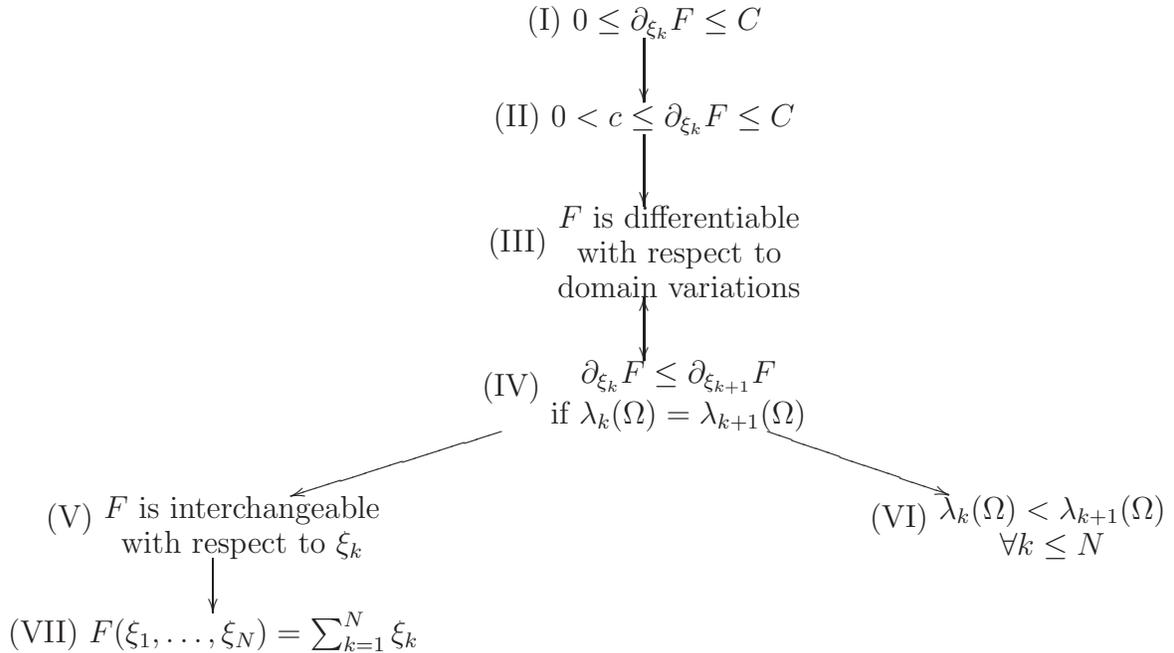
\[
\xymatrix@M=0pt{ & \text{(I) } 0\leq \p_{\xi_k} F \leq C \ar[d]\\
	& \text{(II) } 0<c \leq \p_{\xi_k} F \leq C \ar[d]\\
	& \text{(III) } \txt{$F$ is differentiable \\ with respect to \\ domain variations} \ar@{<->}[d]\\
	& \text{(IV) } \txt{$\p_{\xi_k}F \leq \p_{\xi_{k+1}}F$\\  if  $\l_k(\W) = \l_{k+1}(\W)$} \ar[dl] \ar[dr]\\
	\text{(V) }  \txt{$F$ is interchangeable\\ with respect to $\xi_k$} \ar[d] & & \text{(VI) } \txt{$\l_k(\W)<\l_{k+1}(\W)$\\ $\forall k\leq N$}\\
	\text{(VII) } F(\xi_1,\ldots,\xi_N) = \sum_{k=1}^N \xi_k
	}
\]
We assume (II) in Theorem \ref{thm:shapes}; we obtain a better estimate on the singular set if we also assume (VI) (see \ref{thm:sing}). In \cite{MTV}, (VII) is assumed, and their estimate on the singular set is the same as ours under (VI); it is clear that their proof is only using (V). The argument in \cite{MTV}, and our singular set estimate, should really only require (III), although neither work pursues this explicitly. It is our belief that the differences between (III),(II), and (I) are profound, with each step requiring additional arguments at the free boundary and/or shape optimization level, and also entailing genuinely different behavior for the minimizers. Under (III), the free boundary condition can be obtained explicitly in terms of the derivatives of $F$, by performing domain variations; this means that the domain turns out to be an "almost minimizer" of the corresponding free boundary problem. Under (II), we obtain a free boundary condition of the same form, but now without any explicit formula for the coefficients, by an approximation argument; here the domain should be thought of as "stationary" for the corresponding free boundary problem. Our \emph{free boundary} result is stronger in that it does not assume that any of the $(u_k)_\nu$ have a sign, and that suggests an application to (I). We plan on discussing this in future work.

The paper is organized as follows: Section \ref{sec:general} contains various measure-theoretic properties of minimizers, most notably the linear growth of the first eigenfunction. In Section \ref{sec:blowup}, we present a sequence of blow-up arguments that allow us to study the tangent objects at points of the reduced boundary, obtain good starting configurations, and interpret condition \eqref{eq:two} in a useful way. Then in Section \ref{sec:propE}, we show that all minimizers have a free boundary condition by an approximation argument. The proof of the regularity of the reduced boundary occupies us for Sections \ref{sec:harnack}, \ref{sec:flat}, and \ref{sec:iteration}, while in Section \ref{sec:sing} we briefly discuss the size of the singular set.

We will write $F(\W)$ for $F(\l_1(\W),\ldots, \l_N(\W))$, but take partial derivatives with respect to the eigenvalues regardless. The expressions $x^+$ and $x^-$ refer to the positive and negative parts, respectively, with the convention that $x^-\geq 0$. The index $k$ will typically be reserved for enumerating eigenvalues and eigenfunctions.

\section{General Properties}\label{sec:general}

It will be useful for technical reasons to consider the more general variational problem
\[
 F(\W) + \int_\W f +\int_{\R^n\sm \W} g,
\]
where $f$ is a Lipschitz function with $|f-1|\leq \eta$, and $g$ is a Lipschitz function supported on $B_{R_0}$, with $|g|\leq \eta$, for some small but fixed $\eta$ and large but fixed $R_0$. A set $\W$ minimizing the above over all quasiopen sets will be called a \emph{minimizer}. A set $\W$ minimizing the above over all quasiopen competitors $\W'$ with $\W'\triangle \W\ss U$ will be called a \emph{local minimizer on} $U$.

We say a constant is \emph{universal} if it depends only on a bound on $|\n F|\leq \l^{-1}$, a lower bound on $\p_{\xi_k}F \geq \l$, and on $\eta, R_0$. Note that for a minimizer, $F(\W)$ is bounded by a universal quantity (use $B_1$ as a competitor), and so $|\W|,\l_k(\W)$ are as well.  In this section, it will be useful to consider Lipschitz $F$, which satisfy the above bounds in the sense
\[
\l (\l_k - \l_k')\leq F(\l_1,\ldots,\l_k,\ldots,\l_N) - F(\l_1,\ldots, \l_k',\ldots, \l_N)\leq \l^{-1} (\l_k - \l_k')
\]
for any $\l_k>\l_k'$. In subsequent sections, we will only consider $C^1$ functionals $F$.

We claim that the results of \cite{BMPV} are valid for our minimizers (indeed, Theorem 5.6 applies directly to our functional).  In particular, we will use freely that $|\n u_k|\leq C$ for a universal $C$.  We now present a series of lemmas aimed at showing the linear growth of the  eigenfunctions and related properties.

\begin{lemma}\label{lem:lowerbd} Let $\W$ be a local minimizer on $B_r(x)$. Then there are $c_0,r_0>0$ depending only on $\l_N$ such that if $|B_{r/8}(x)\cap \W|>0$  and $r<r_0$, then
\[
 \sup_{B_{r}(x)} \sum_{k=1}^N|u_k| \geq c_0 r.
\]
\end{lemma}

This lemma was originally proved by Alt and Caffarelli for the one-phase Bernoulli-type problem \cite{AC}. The simple and more robust argument presented here is taken from David and Toro \cite{DT}.

\begin{proof}
Assume that for some $r<r_0$ (with $r_0,c_0$ to be chosen) we have 
\[
 \sup_{B_{r}(x)} \sum_{k=1}^N|u_k| \leq c_0 r.
\]
Translating, take $x=0$. We will show that for any $z\in B_{1/4}$, we have that
\[
 \sup_{B_{r/4}(z)} \sum_{k=1}^N|u_k| \leq \frac{c_0 r}{4}.
\]
Then applying this claim inductively to balls of radii $4^{-i}r$ centered on each $z\in B_{1/4}$, we learn that $|u_k|$ all vanish on $B_{r/4}$. This is a contradiction.

Let us then establish this claim. First, on $B_r$ we have
\[
 \int_{B_r}|u_k|^2 \leq C c_0^2 r^{n+2}.
\]
Applying the Cacciopoli inequality,
\[
 \int_{B_{9r/10}} |\n u_k|^2 \leq C\l_N^2 c_0^2 r^{n+2} +  Cr^{-2} \int_{B_r}|u_k|^2 \leq C c_0^2 r^n
\]
provided $r_0\leq 1$. The first term is from $\triangle |u_k|\geq -c_0 r \l_N$.  Next, take $\W' = \W \sm B_{3r/4}$ as a competitor for $\W$. Then
\[
 |\W \cap B_{3r/4}| \leq C(F(\W') - F(\W)).
\]
To estimate the latter quantity, take a smooth, radially increasing cutoff function $\phi$ which vanishes on $B_{3r/4}$ and is $1$ outside $B_{9r/10}$, and has derivative bounded by $C r^{-1}$. Let $v_k = \phi u_k$: then
\[
 \int |v_k -u_k|^2 \leq C c_0^2 r^{n+2},
\]
\[
 \int |\n v_k -\n u_k|^2 \leq C c_0^2 r^n,
\]
and
\[
 |\int \n v_k \cdot \n v_j|\leq C c_0^2 r^n.
\]
It follows that
\begin{align*}
 \l_k(\W')& = \inf_{E\ss H^1_0(\W'), \dim E=k} \sup_{v\in E,v\neq 0} \frac{\int |\n v|^2}{\int v^2}\\
& \leq \sup_{v = \sum_{j=1}^k \a_j v_j} \frac{\int |\n v|^2}{\int v^2}\\
& \leq \l_k(\W) + C c_0^2 r^n.
\end{align*}
Using that $|\n F|\leq C$, we have
\[
 |\W \cap B_{3r/4}|\leq C c_0^2 r^n.
\]
Finally, take any $y\in B_{r/2}$ and use that $\triangle |u_k| \geq - c_0 r \l_N$:
\begin{align*}
 \sum_{k=1}^N |u_k(y)|&\leq \sum_{k=1}^N \fint_{B_{r/4}(y)}|u_k| + C c_0 \l_N r^3\\
&\leq C c_0 r\sum_{k=1}^N \frac{|B_{r/4}(y)\cap \W|}{r^n} + C c_0 r^3\\
& \leq C( c_0^2 +r^2)c_0 r.
\end{align*}
Choosing $c_0,r_0$ so that $C(c_0^2 + r_0^2) \leq 1/4$ implies our claim.
\end{proof}

Note that a minimizer has at most $N$ connected components; if it had more, then removing the one with the highest first eigenvalue $\l_1$ (which does not change the eigenvalues $\l_1,\ldots, \l_N$ used to compute the functional, but reduces volume) would lead to a contradiction.

\begin{corollary}\label{cor:diam} Let $\W$ be a minimizer. Then for every $x\in \p \W$ and $r<r_0$, we have 
	\[
	|B_r \cap \W|\geq c_0 r^n.
	\]
	 As a consequence, each connected component $V$ of $\W$ has universally bounded diameter.
\end{corollary}

\begin{proof}
	Let $r_0$ be as in Lemma \ref{lem:lowerbd}, and use that lemma applied to $B_{r/2}(x)$ to find $y \in B_{r/2}(x)$ and $k$ such that $|u_k|(y) > c_0 r/2$. Together with the Lipschitz property, this means $|u_k|\geq c_0 r /4$ on $B_{\a r/2}(y)$ for some universal $\a<1$, and so in particular $B_{\a r/2}(y)\ss \W$. This implies the first property.
	
	Now assume that for some $K>0$ there is a minimizer which admits $K$ disjoint balls $B_{r_0}(x_i)$ with $x_i \in \p \W$. Then for each of these balls, we have $|\W \cap B_[r_0](x_i)|\geq c_0 r_0^n=c$, and so $|\W|\geq c K$. As the volume of $\W$ is universally bounded, so is $K$, which implies the second property.
\end{proof}

\begin{remark}\label{rem:disconlb}
	It is easy to check that if $\W$ is a disconnected minimizer, the proof in Lemma \ref{lem:lowerbd} can be applied to each connected component $V$, giving
	\[
	\sup_{B_r \cap V}|u_k|\geq c_0 r.
	\] 
	Corollary \ref{cor:diam} may be adapted similarly.
\end{remark}

The following is another classic argument of Alt and Caffarelli:

\begin{lemma}\label{lem:mink} For a minimizer $\W$, the boundary $\p \W$ has universally bounded (upper) Minkowski content: there are universal constants $C,r_0$ such that for all $r\in (0,r_0)$,
 \[
  |\{x:d(x,\p \W)<r\}|\leq C r.
 \]
 Moreover, the local estimate
  \[
  |B_R(x)\cap  \{x:d(x,\p \W)<r\}|\leq C r R^{n-1}.
 \]
 also holds for $x\in \p \W$ and $R<r_0$.
\end{lemma}

Note that the upper Minkowski content controls $\cH^{n-1}$, so the local bound implies 
\[
\cH^{n-1}(B_R(x))\leq CR^{n-1}.
\]

\begin{proof}

The functions $|u_k|$ are Lipschitz, and satisfy $\triangle |u_k| \geq - \l_k |u_k|$ in the distributional sense. In particular, $\triangle  |u_k| + \l_k |u_k| d\cL^n$ is a nonnegative Borel measure supported on $\p\{|u_k|>0\}$. Set $\nu$ to be the measure
 \[
  \nu = \sum_{k=1}^N \triangle |u_k| + \l_k |u_k|d\cL^n.
 \]
We will now show that for any $x\in \p \W$ and $r<r_0$, we have
\[
 c r^{n-1} \leq \nu(B_r(x)) \leq C r^{n-1}.
\]

We may set $x=0$. For the upper bound, we estimate (for $\phi(x) = (2-|x|/r)_+$):
\[
 \nu(B_r) \leq  \sum_{k=1}^N \int \phi d\triangle |u_k| + \int \l_k \phi |u_k| = - \sum_{k=1}^N \int \n\phi \cdot \n |u_k| + \l_k |u_k|\phi\leq C r^{n-1},
\]
using that $|\n |u_k||= |\n u_k|\leq C$ almost everywhere.

We now show the lower bound. For some $y\in B_{r/4}$ and $k$, $|u_k(y)|\geq c_0 r$ for a universal $c_0$. Set $s=d(y,\p \{u_k=0\})\in (c r, r/4]$ and find $\xi$ with $|\xi-y|=s$ and $u_k(\xi)=0$. By the Hopf lemma and the Lipschitz estimate, we have that $c\min\{|\xi-z|,s-|y-z|\}\leq |u_k(z)|\leq C |\xi -z|$ for $z\in B_s(y)$. Let $G_z$ be the positive Green's function for $B_{\r}(\xi)$ with pole at $z$, with $z\in B_{s}(y)$. As $|u_k|$ is positive there, we have for almost every $\r\in (s/2,s)$,
\[
 \int_{B_\r(\xi)} G_z [d\triangle |u_k| +\l_k |u_k|d\cL^n] = -u(z) + \l_k \int_{B_\r(\xi)}G_z |u_k| - \int_{\p B_\r(\xi)} u \p_r G_z d\cH^{n-1}.   
\]
Fixing such a $\r$, select a $z$ with $c\k r\leq |u_k|(z) \leq C \k r$ for a small $\k$. Then $-\p_r G_z \geq c/r^{n-1}$, and so
\[
 -u(z) + \l_k \int_{B_\r(\xi)}G_z |u_k| - \int_{\p B_\r(\xi)} u \p_r G_z d\cH^{n-1}  
\geq - C\k r + \frac{c}{r^{n-1}} \int_{\p B_\r(\xi)} |u_k|d\cH^{n-1} \geq c r 
\]
by choosing $\k$ small enough. We used here that $u\geq c\r$ on a region of area $c\r^{n-1}$ of $\p B_\r(\xi)$ for any $\r$.

On the other hand, the measure $\triangle |u_k| +\l_k |u_k|d\cL^n$ is supported outside of $\{u_k\neq 0\}$, and hence outside of a ball $B_{c \k r}(z)$. Outside of this ball $G_z \leq C (\k r)^{2-n} $, and hence
\[
 \int_{B_\r(\xi)} G_z [d\triangle |u_k| +\l_k |u_k|d\cL^n] \leq C r^{2-n} \int_{B_\r(\xi)\sm B_{c\k \r}(z)} [d\triangle |u_k| +\l_k |u_k|d\cL^n]\leq Cr^{2-n}\nu (B_r).
\]
Combining, this gives the lower bound.

We now show the conclusion of the lemma. Set $V_r = \{x: d(x,\p \W)<r\}$. Take a cover of $\p \W$ by balls $B_{r}(x_i)$, with $r<r_0$ fixed, where no more than $C(n)$ balls overlap. Then if there are $M_r$ balls in this cover, 
\[
 M_r r^{n-1}\leq C \sum_{i=1}^{M_r} \nu(B_r(x_i)) \leq C(n)\nu(\bigcup_i B_r(x_i))\leq C(n) \nu (V_r).
\]
On the other hand, $V_r \ss \cup_i B_{2r}(x_i)$, so
\[
 \nu (V_r)\leq \sum_{i=1}^{M_r} \nu(B_{2r}(x_i)) \leq C M_r r^{n-1}.
\]
Applying this second estimate with $r_0$, and using that $M_{r_0}$ is universally bounded from Corollary \ref{cor:diam}, we have $M_r r^{n-1}\leq C$. Finally,
\[
 |V_r|\leq \sum_{i=1}^M |B_{2r}(x_i)| \leq C M_r r^n \leq C r
\]
for any $r<r_0$.

The local bound may be proved in the same way, now using $\nu(B_R\cap V_r)\leq CR^{n-1}$ for the upper estimate.
\end{proof}

In the following lemma, a priori the constants may depend on non-universal quantities. It should be thought of, at first, as a qualitative property. We will later show that for $\eta$ small, the constants are actually universal. 

Let
\[
 \bar{u}=\sum_{i=1}^N |u_k|;
\]
this is a Lipschitz function with $c d(x,\p \W) \leq \bar{u} \leq C d(x,\p \W)$, and $-\triangle \bar{u}\leq \l_N \bar{u}$. If a minimizer $\W$ has several connected components, the first eigenvalue of each of them must number among $\l_1(\W), \ldots, \l_N(\W)$, and we may select a basis of eigenfunctions so that for each component $V$, one of the $u_k$ is nonnegative, supported on $V$, and is the first eigenfunction of $V$ (extended by $0$). This special eigenfunction will be denoted $u_V$, with eigenvalue $\l_V$.

\begin{lemma}\label{lem:stronglb} There is a universal constant $r_0$ such that the following holds: Let $V$ be a connected component of a minimizer $\W$ and $u_V$ be the special eigenfunction as above. Then  the value
\[
A_V=\frac{r_0}{\inf_{x\in V: \bar{u}=r_0}u_V}<\8,
\]
and there is a constant $A=A(A_V)$ such that $\bar{u}\leq A u_V$ on $V$. As a consequence, for any $x\in \p V$ and $r<r_0$,
\[
 \d<\frac{|B_r(x)\cap \W|}{|B_r|}<1-\d
\]
and
\[
 \d < \frac{\cH^{n-1}(B_r(x)\cap \p \W)}{\cH^{n-1}(\p B_r)} <\d^{-1}.
\]
Here $\d=\d(\max_{V} A_V)$.
\end{lemma}

\begin{proof}
 First, we note that $\max u_V$ is bounded from below by a universal constant $c_*$. Indeed,
 \[
  1 = \int_{V} u_V^2 \leq |V|\max u_V^2 \leq C\max u_V^2.
 \]
For an $r_0$ to be determined shortly, let 
\[
A_V=\frac{r_0}{\inf_{x\in V: \bar{u}=r_0}u_V}. 
\]
Provided $r_0< c_*$, this is a well-defined quantity. Set $v_0 = A_V u_V - \bar{u}$; this function satisfies $-\triangle v_0 \geq A_V \l_V u_V - \l_N |u|$. Finally, let $V_r = \{x\in V: 0<\bar{u}(x)<r\}$, which has volume $|V_r|\leq Cr$ by Lemma \ref{lem:mink}.

On $\p V_{r_0}$, $v_0\geq 0$ by construction. Moreover, on $V_{r_0}$ we have $-\triangle v \geq -C r_0$. Applying the Cabr\'{e} maximum principle \cite{Cabre}, this gives
\[
 -\inf_{V_{r_0}} v_0 \leq C r_0 |V_{r_0}|^{1/n} \leq C_* r_0^{1+1/n},
\]
where $C_*$ is a universal constant. In particular, on $V_{r_0}\sm V_{r_0/2}$ we have
\[
 \frac{r_0}{2}\leq \bar{u} \leq A_V u_V + C_* r_0^{1+1/n}.
\]
Choose $r_0$ so that $r_0< (4C_*)^{1/n}$ and set
\[
 A_k = \frac{A_{k-1}}{1-2C_* \1\frac{r_0}{2^{k-1}}\2^{\frac{1}{n}}}.
\]
Defining $v_k = A_k u_V -\bar{u}$, we have just shown that if $v_{k-1}\geq 0$ on $\{\bar{u}=2^{1-k} r_0\}$, then $v_k \geq 0$ on $V_{2^{1-k}r_0}\sm V_{2^{-k}r_0}$. Now,
\[
 \log A_k = \log A_V - \sum_{j=1}^k \log (1-2C_* \1\frac{r_0}{2^{k-1}}\2^{\frac{1}{n}}) \leq \log A_V + C\sum_{j=1}^k 2^{-j/n}\leq C+\log A_V,
\]
so $A_k\leq C A_V$ for all $k$. We thus have
\[
 \bar{u}\leq A_k u_V \leq CA_V u_V
\]
on $V_{2^{1-k}r_0}\sm V_{2^{-k}r_0}$, and hence the whole of $V_{r_0}$.

We now discuss the remaining conclusions, setting $x=0$. The lower bound on $|B_r \cap \W|/|B_r|$ was already shown earlier in Corollary \ref{cor:diam}. For the upper bound, we use $\W' = \W \cup B_s$ as a competitor for $\W$, with $s<r$ to be chosen shortly. As $\W\ss \W'$, we have $\l_k(\W')\leq \l_k(\W)$ for every $k$. Now, among all of the components $V$ which intersect $B_s$, select $V_*$ to be the one with the lowest first eigenvalue $\l_1 (V)$. Let $h$ be a function with the same trace as $u_{V_*}$ on $\p B_s$ and harmonic on $B_s$, and $w$ be the $H^1$ function given by
\[
w = \begin{cases}
	u_{V_*} & x\in V_*\sm B_s\\
	h & x\in B_s\\
	0 & \text{otherwise}.
\end{cases}
\]
Using $w$ as a competitor in the minimization formula for the first eigenvalue of the connected component of $\W'$ which contains $B_s$ (denoted by $\l'$), we see that
\[
\l' \leq  \l_{V_*}  - [ \int_{B_s}|\n u_{V_*}|^2 - |\n h|^2 ] + Cs^{n+2},
\]
where the final term is from bounding the error in normalization. It then follows from the minimality of $\W$ that
\[
 \int_{B_s}|\n u_{V_*}|^2 - |\n h|^2   \leq C|B_s\sm \W| + Cs^{n+2}.
\]
 Using the Poincar\'{e} inequality,
\[
 \frac{1}{s^2}\int_{B_s}|u_V - h|^2 \leq  \int_{B_s}|\n u_V - \n h|^2 \leq C|B_s\sm \W| +Cs^{n+2}.
\]
Now, there is a point $z$ in $B_r$ where $u_V\geq c A_{V_*}^{-1} r$ (this is from the just-established claim and Remark \ref{rem:disconlb}); we then must have $|z|=s\geq cr$ and $\int_{\p B_s}u_V \geq c r^n$. Then $h(0)\geq cr$, and indeed $h\geq c/2 r$ on a small ball $B_{\k r}$. By contrast, $u_\leq C\k r$ on that ball, and so
\[
 \k^n r^{n-2} (\frac{c}{2} -\k C)^2 r^2 \leq C|B_r\sm \W| + Cr^{n+2}.
\]
Choosing $\k$ so that the term in parentheses is positive and then $r<r_0$ small enough, this gives the bound.

The upper bound on $\cH^{n-1}(B_r\cap \p \W)$ was already proved in Lemma \ref{lem:mink}, and the lower bound is now a consequence of the relative isoperimetric inequality.
\end{proof}

One helpful consequence of the above lemma is that it implies that $\cH^{n-1}(\p \W\sm \p^*\W)=0$, where $\p^*\W$ is the reduced boundary. From the argument in Lemma \ref{lem:mink}, we know that $\nu_k=-\triangle u_k - \l_k u_k$ is a measure supported on $\p \W$ with $|\nu_k (B_r)(x)|\leq Cr^{n-1}$ for all $x\in \p \W$ and $r$ small enough. We now see that this implies that $\nu_k\ll \cH^{n-1}\mres \p \W$, and from the Radon-Nikodym theorem and this estimate
\[
\nu_k = (u_k)_\nu d\cH^{n-1}\mres \p \W
\]
for some bounded Borel function $(u_k)_\nu$.

We emphasized in the above proof that all constants depend only on universal quantities and the values $A_V$. Whether $A_V$ can be taken to be universal is unclear, and is related to the question of whether there may be disconnected minimizers of some functional $F$ satisfying our hypotheses. However, at least when $\eta$ is small enough, $A_V$ is a universal quantity. This is the aim of the next series of lemmas. 

\begin{lemma} \label{lem:connected} Let $\eta=0$ and $\W$ be a minimizer. Then $\W$ is connected, and $A_\W$ is universally bounded.
\end{lemma}

\begin{proof}
	We first show that $\W$ is connected. Indeed, assume not; then up to translating some of the components (here we use that $\eta=0$ and the functional is translation-invariant), we may assume two components $\W_1,\W_2$ have a common boundary point at $0$. Let $v_1$, $v_2$ be the (nonnegative) first eigenfunctions for $\W_1,\W_2$ respectively, which we know satisfy $\sup_{B_r} v_i \geq c A_i r$ for some domain-specific $A_i^{-1}>0$, $|\n v_i|\leq C$, and $c d(x,\p \W_i) \leq v_i(x) \leq C  d(x,\p \W_i)$. Taking the blow-up sequence
	\[
	 w_i^r = \frac{v_i(r \cdot)}{r},
	\]
	we see that these converge locally uniformly on their positivity sets (along subsequences) to a pair of harmonic functions $w_1,w_2$ enjoying the same three properties. Moreover, the sets $\{w_i^r >0\}$ converge in $L^1_{\text{loc}}$ to $\{w_i>0\}$: if $x\in \{w_i>0\}$, then $w_i^r(x)>0$ for $r$ large enough, and $x\in \{w_i^r(x)>0\}$. If  $x\in \bar{\{w_i > 0\}}^C$, then $w_i(y)=0$ on a neighborhood $B_s(x)$, and for $r$ large $w_i^r\ll s$ on $B_s(x)$; from the nondegeneracy property it follows that $w_i^r(x)=0$. Finally, we have that on $\p \{w_i>0\}$, $w_i$ is not differentiable: if $\n w_i(x)=0$ for an $x \in \p \{w_i>0\}$, this is a contradiction to the nondegeneracy property on sufficiently small balls around $x$, while if $|\n w_i(x)|>0$, this will contradict the positivity of $w_i$. Using that $w_i$ is Lipschitz continuous and Rademacher's theorem, we have $|\p \{w_i>0\}|=0$, and so $1_{\{w_i^r>0\}}\rightarrow 1_{\{w_i>0\}}$ almost everywhere, and so in $L^1_{\text{loc}}$ using the dominated convergence theorem.
	
	This means the $w_i$ have disjoint support, and yet $|B_s \cap \{w_1=w_2=0\}|/|B_s| >\d$ by using Lemma \ref{lem:stronglb}. This, however, is in direct contradiction of the Alt-Caffarelli-Friedman monotonicity formula: if
	\[
	J(r) = \frac{1}{r^4}\int_{B_r} \frac{|\n w_1|^2}{|x|^{n-2}}\int_{B_r} \frac{|\n w_2|^2}{|x|^{n-2}},
	\]
	we have that $J(r)\leq J(1)r^\a \rightarrow 0$ (see \cite[Corollary 12.4]{CS}). On the other hand,
	\[
	\frac{1}{r^2}\int_{B_r} \frac{|\n w_1|^2}{|x|^{n-2}} \geq \frac{1}{r^n}\int_{B_r} |\n w_1|^2 \geq c \frac{1}{r^{n+2}} \int_{B_r} w_1^2 \geq c,
	\]
	and likewise for $w_2$.
	
	Next we show that $A_\W$ is universally bounded. If this were not the case, we could produce a sequence of minimizers $\W_i$ and functionals $F_i$ with increasing values $A_i$. As we know that the diameters of $\W_i$ are bounded universally, we may assume that after a translation the $\W_i$ all reside in some large ball $B_R$. Extracting a subsequence, we may assume that $\l_k(\W_i) \rightarrow \l_k$, and the eigenfunctions $u_k^i\rightarrow u_k$ uniformly. If we set $\W$ to be the union of the supports of the $u_k$, we have $\W_i \rightarrow \W$ in the Hausdorff topology and in $L^1$ (like in the argument above), and the $u_k$ satisfy the equations $-\triangle u_k = \l_k u_k$ on $\W$. Finally, take $F_i \rightarrow F$ uniformly for some $F$; this $F$ will still have the same properties. For any fixed open set $U$, it is clear that $\lim F_i(U) = F(U)$, so we must have
	\[
	\inf_{U \text{open}} F(U) +|U| \geq \lim_i \inf_{U \text{open}} F_i(U) +|U| = \lim_i F_i(\W_i)+|\W_i| = F(\l_1,\ldots,\l_N) + |\W|.
	\]
	On the other hand, $\l_k(\W)\leq \l_k$ (simply using that $u_k$ are orthonormal eigenfunctions on $\W$), and so
	\[
	\inf_{U \text{open}} F(U) +|U|\leq F(\W) +|\W|\leq F(\l_1,\ldots, \l_N) +|\W|. 
	\]
	We conclude that $\W$ is a minimizer for $F$. Hence, $\W$ is connected. The argument above also implies that $\l_k = \l_k(\W)$, and in particular $u_1>0$ on $\W$.
	
	On the other hand, $A_i \rightarrow \8$. This means that on $\{|\bar{u}^i|=r_0\}$, we have $\inf u_1^i \rightarrow 0$; in other words, there is a sequence $x_i\rightarrow x\in \W$ with $u_1^i(x_i)\rightarrow 0$. It follows that $u_1(x)=0$, and this is a contradiction.
\end{proof}

Recall that $R_0$ is the radius on which $g$ is supported.

\begin{lemma} \label{lem:con2} There is a constant $\eta_0=\eta_0(R_0,\l,n, N)$ such that if $\eta<\eta_0$, minimizers are connected and $A_0$ is uniformly bounded. 
\end{lemma}

\begin{proof} Assume for contradiction that one of the two conclusions fails. Then there is a sequence of functionals $F_\eta$ with $\eta \rightarrow 0$ and minimizers $\W_\eta$ which violate that conclusion. Arguing exactly as in the proof of Lemma \ref{lem:connected}, we have $F_\eta\rightarrow F$ uniformly, $\W_\eta \rightarrow \W$ in Hausdorff topology and in $L^1$, $\l_k(\W_\eta) \rightarrow \l_k(\W)$ for $k=1,\ldots,N$, and $\W$ is a minimizer for $F$. Note that we have used that if a connected component of $\W_\eta$ is outside of $B_{R_0}$, we may translate it, so that the diameters of $\W_\eta$ remain bounded.
	
Say that the $\W_\eta$ violated the first conclusion: they were disconnected. Then for infinitely many $\eta$, there are two eigenfunctions $u_1^\eta,u_k^\eta$ among the first $N$ which have disjoint support. However, by Lemma \ref{lem:connected}, $\W$ is connected, and so the support of their limits must overlap--this is a contradiction.

If instead $A_\eta \rightarrow \8$, we obtain a contradiction like in the proof of \ref{lem:connected}, as this implies that $u_1$ vanishes at some point in $\W$.
\end{proof}

From now on, we will always assume that $\eta<\eta_0$, so that the conclusions of Lemma \ref{lem:con2} apply.

\section{Blow-ups}\label{sec:blowup}

Below, a minimizer has property S if $\l_1<\ldots <\l_N<\l_{N+1}$, i.e. all of the eigenvalues are simple. It has property E if 
\[
 \sum_{k=1}^N (u_k)_\nu^2(x) \xi_k = \xi_0(x) \qquad \cH^{n-1}-\text{a.e. on }\p\W
\]
for some nonnegative constants $\{\xi_k\}_{k=1}^N$ and $\xi_0(x)=f(x)-g(x)$, and some choice of orthonormal eigenfunctions $u_k$. Note that it is possible to obtain a basis-independent version of property E, which would read that given an orthonormal basis of eigenfunctions $\{u_k\}$ there is a constant, symmetric, positive semidefinite matrix $A$ such that
\[
 u_\nu^T A u_\nu = \xi_0(x),
\]
where $u$ is the vector of eigenfunctions, and moreover $A$ is block diagonal with a block for each distinct eigenvalue of $\W$. The first form of property E is then recovered by applying the spectral theorem to each block; below, we will always elect to work with a set of eigenfunctions $u_k$ which diagonalizes $A$. 

Property E should be thought of as the stationarity condition for a minimizer. We will shortly establish that property S implies property E, as in that case property E is exactly the Euler-Lagrange equation of $\W$ with respect to domain variations. Later we show that \emph{every} minimizer satisfies property E (with constants $\xi_k$ bounded from above and below); this is a surprising finding, as in that case $F$ need not be differentiable with respect to domain variations, and so property E gives more refined information on the geometry of $\W$. We do not believe that all minimizers must satisfy property S.

\begin{lemma} \label{lem:domvar} Let $\W$ be a minimizer. Then Property $E$ is equivalent to the following fact: for any vector field $\U$ supported on a ball $B_r(x)$, $x\in \p \W$, with $\|\U\|_{C^1(\R^n;\R^n)}\leq 1$, let $\phi_t$ be the corresponding family of diffeomorphisms
	\[
	\phi_t(x) = x + t \U(x)
	\]
	with $|t|\leq 1$. Then if $u_k^t = u_k \circ \phi_t$,
	\begin{equation}\label{eq:Edomvar}
	 \left| \int_\W \sum_k \xi_k |\n u_k|^2 + f - g  - \int_{\phi_t^{-1} (\W)} \sum_k \xi_k |\n u_k^t|^2 + f -g \right| \leq C(t^2 + t r^2) r^n \max |\n \U|, 
	\end{equation}
	where $C$ is universal.
\end{lemma}

\begin{proof}
	Let us begin with a technical aside. Consider the functional
	\[
	S: \U \mapsto \int_\W  |\n u_k|^2 \dvg \U - 2 \n u_k \n \U \n u_k + 2 \l_k u_k \U \n u_k
	\]
	defined on compactly supported $C^1$ vector fields $\U$. Now, for any $\U$ supported away from $\p \W$, it is easy to check by integrating by parts that $S$ vanishes. For any $\U$ with one component $\U^1 \geq 0$,  $\U^2,\ldots,\U^n=0$  take the functions
	\[
	\U_\d(x) = e_1 (\max_{y\in \p \W} \{ \U^1 (y) - |x-y|/\d \})^+,
	\]
	which for $\d<1/\max |\n \U|$ coincide with $\U$ on $\p \W$ and have $|\n \U_\d|\leq 1/\d$. We have then that $S(\U)=S(\U_\d)$, and (using the Minkowski content bound of Lemma \ref{lem:mink})
		\[
		|S(\U_\d)| \leq C \int_{d(x,\p \W)< \d \max |\U|} |\n \U_\d| + |\U_\d| \leq C \max |\U|.
		\]   
	We may do this for the positive and negative parts of each component of $\U$, giving $|S(\U)|\leq C\max |\U|$ for any $\U$ by linearity. Applying the Riesz representation theorem, we learn that $S$ is represented by a vector-valued Borel measure $s$ supported on $\p \W$. Another application of the argument just made reveals that $|s|(B_r(x))\leq Cr^{n-1}$ for each $x\in \p \W$, and hence $s$ admits a representation
	\[
	s(E) = \int_{E\cap \p \W} g d \cH^{n-1}.
	\]
	We claim that the bounded density $g$ is given by $g = - (u_k)^2_\nu \nu$ (where $\nu$ represents the measure-theoretic outward normal at points in $\p^* \W$). It suffices to check this at each point of the reduced boundary where
	\[
	\fint_{B_r(x)\cap \p \W} |g - g(x)|d \cH^{n-1} \rightarrow 0
	\]
	and
	\[
	 \fint_{B_r(x)\cap \p \W} |(u_k)_\nu - (u_k)_\nu(x)|d \cH^{n-1} \rightarrow 0.
	\]
	At such a point, perform a blow-up along a subsequence $r_i\searrow 0$
	\[
	v_i = \frac{u_k(x + r_i \cdot)}{r_i} \rightarrow v,
	\]
	with the convergence to a continuous $v$ guaranteed by the Lipschitz bound on $u_k$. As $x\in \p^* \W$, we have that $v$ vanishes on a half-space, is harmonic on the complementary half-space, and has $|\n v|\leq C$. It follows that in coordinates with $e_n = \nu$, $v = \a x_n^-$. Moreover, from the distributional convergence of $\triangle v_i\rightarrow \triangle v$, we have that $\a = - (u_k)_\nu (x)$.
	
	On the other hand, we have that for each vector field $\U$,
	\[
	\int_{(\W-x) /r_i}  |\n v_i|^2 \dvg \U - 2 \n v_i \n \U \n v_i + 2 \l_k r_i^2 v_i \U \n v_i = \int_{(\p \W-x)/r_i} g(x + r_i \cdot) \U d\cH^{n-1}. 
	\]
	 Taking the limit on each side (using dominated convergence theorem) gives
	 \[
	  \int_{x_n <0} |\n v|^2 \dvg \U - 2 \n v \n \U \n v = \int_{x_n = 0} g(x)\U d\cH^{n-1}.
	 \]
	 Entering our representation for $v$ and integrating by parts gives $g(x) = - (u_k)_\nu^2(x) e_n$, as claimed.
	 
	 We return to the proof of the lemma. Note that under the assumptions made on $\U$, we have $|\U|\leq r$ and $\phi_t^{-1}(x)=x$ outside $B_{2r}(x)$. Using the area formula,
	 \[
	 \int_{\phi_t^{-1} (\W)} |\n u_k^t|^2 = \int_{\W} |\n u_k \n \phi_t\circ \phi^{-1}_t |^2 |\det \n \phi_t^{-1}|.
	 \]
	 Using $\n \phi_t = I + t \n \U$ and expanding in a Taylor series,
	 \[
	 \left|\int_{\phi_t^{-1} (\W)} |\n u_k^t|^2 - \int_{\W}|\n u_k|^2 + 2 t \n u_k \n \U \n u_k - t \dvg \U |\n u_k|^2 \right| \leq Ct^2 r^n \max |\n \U|.
	 \]
	 Now using our representation for $S$,
	 \begin{align*}
	 & \left|\int_{\phi_t^{-1} (\W)} |\n u_k^t|^2 - \int_{\p^* \W}(u_k)_\nu^2 \U \nu d\cH^{n-1} - \int_{\W}|\n u_k|^2 + 2 t \l_k u_k \U \n u_k  \right| \\
	 &=\left|\int_{\phi_t^{-1} (\W)} |\n u_k^t|^2 + S(u)- \int_{\W}|\n u_k|^2 + 2 t \l_k u_k \U \n u_k  \right|\\
	 &\leq C t^2 r^n\max |\n \U|.
	 \end{align*}
	 Similarly, using that $\W$ has finite perimeter,
	 \begin{align*}
	 & |\int_{\phi_t^{-1} (\W)} f - g  + t\int_{\p^* \W} (f - g)\U \nu d\cH^{n-1} - \int_\W f-g|\\
	 & = |\int_{\W} [(f - g)\circ\phi_t^{-1} |\det \n \phi^{-1}_t| - (f-g)]   + t\int_{\W} \dvg (f-g)\U|\\
	 & \leq Ct^2 r^n \max |\n \U|
	 \end{align*}
	 Combining,
	 \begin{align*}
	 &\big| \int_\W \sum_k \xi_k |\n u_k|^2 + f - g  - \int_{\phi_t^{-1} (\W)} \sum_k \xi_k |\n u_k^t|^2 + f -g \\
	 &\qquad + t\int_{\p^* \W}(\sum_k \xi_k (u_k)_\nu^2 - \xi_0) \U\nu d\cH^{n-1}\big|\\
	 &\leq C(t^2 + r^2 t) r^n\max|\n \U|.
	 \end{align*}
	 
	 If property E holds, the surface term vanishes, and so we have just established \eqref{eq:Edomvar}. Conversely, say that \eqref{eq:Edomvar} holds: then we have
	 \[
	 |\int_{\p^* \W}(\sum_k \xi_k (u_k)_\nu^2 - \xi_0) \U\nu d\cH^{n-1}|\leq C(t + r^2) r^n\max|\n \U|.
	 \]
	 Taking $t$ to $0$,
	 \[
	 |\int_{\p^* \W}(\sum_k \xi_k (u_k)_\nu^2 - \xi_0) \U\nu d\cH^{n-1}|\leq C r^{n+2}\max|\n \U|.
	 \]
	 Assuming that at some $x\in \p^* \W$ we have
	 \[
	  \max_k \fint_{\p \W \cap B_r(x)} |(u_k)_\nu - u_k| \rightarrow 0
	 \]
	 and
	 \[
	  |\xi_k - \sum \xi_k (u_k)_\nu^2|>\e, 
	 \]
	 take $q$ a nonnegative smooth function radial about $x$, supported on $B_r(x)$, is $r/4$ on $B_{r/2}$, and has $|\n q|\leq 1$. Set $\U = \nu_x q$, and observe that for $r$ sufficiently small, we have
	 \[
	 |\int_{\p^* \W}(\sum_k \xi_k (u_k)_\nu^2 - \xi_0) \U\nu d\cH^{n-1}| \geq c\e r^{n}.  
	 \]  
	 On the other hand, the same quantity is bounded by $Cr^{n+2}$ by our estimate, which is a contradiction for small enough $r$.
\end{proof}

\begin{lemma}\label{lem:blowup} Let $\W$ be a minimizer, and $0\in \p \W$.
	\begin{enumerate}
		\item Then the blow-up sequence
		\[
		\3u_k^i(\cdot):=\frac{u_k(r_j \cdot)}{r_j}\4
		\]
		admits a subsequence converging locally uniformly to functions $\{v_k\}$, with $v_1\geq 0$. The $v_k$ are harmonic on $\{v_1>0\}$, have $|\n v_k|\leq C$, and vanish on $\{v_1=0\}$.
		\item If property E holds, then
		\[
		\sum_{k=1}^N (v_k)_\nu^2(x) \xi_k = \xi_0(0) \qquad \cH^{n-1}-\text{a.e. on }\p \{v_1>0\},
		\]
		where $(v_k)_\nu$ is a Borel function for which $\triangle v_k = - (v_k)_\nu d\cH^{n-1}\mres \p \{v_1>0\}$.
		\item If property S holds, then the $v_k$ are local minimizers of the functional
		\[
		F_*(\{v_k\};E) = \sum_{k=1}^N\int_E \xi_k |\n v_k|^2  + \xi_0|E|, 
		\]
		where the constants $\xi_k$ are given by $\p_{\l_k} F(\l_1,\cdots, \l_N)$ and $\xi_0 = f(0)-g(0)$ (among open sets $E$ and collections $v_k$ of Lipschitz continuous functions which are supported on $E$).
	\end{enumerate}  
\end{lemma}

\begin{proof} The first property is immediate from the Lipschitz estimate on $u_k$ and the lower bound on $u_1$. The second one can be checked using the equivalent characterization in Lemma \ref{lem:domvar}, which easily passes to the limit and implies (2) by the same argument as is used in the proof. For the final one take any competitor $(E',v_k')$ with $E'-E, v_k'=v_k$ outside $B_R$. Set $\phi$ to be a smooth radially increasing cutoff which vanishes on $B_R$ and is $1$ on $B_{2R}$, and $w_k^i$ to be
\[
 w_k^i=\phi u_k^i + (1-\phi)v_k'.
\]
Let $\W^i=\cup \{w_k^i>0\}$, and use the dilated $\tilde{\W}^i=\W^i/r_i$, $\tilde{w}_k^i = r_i w_k^i(\cdot /r_i)$ as a competitor for $\W$. These have
\[
\int |\n \tilde{w}^i_k - u_k|^2 \leq Cr^n_i
\]
and
\[
\int |\tilde{w}^i_k - u_k|^2 \leq Cr^{n+2}_i.
\]
 We then have that
\begin{align*}
\l_k(\tilde{\W}^i) &\leq \max_{q = \sum_{j=1}^k \a_j \tilde{w}_j^i,\sum_{j=1}^k\a_j^2=1} \frac{\int |\n q|^2}{\int q^2}\\
& \leq  \max_{\a_j} \frac{\sum_{j=1}^k\int \a_j^2 |\n \tilde{w}_j^i|^2 + \sum_{l< j}2\a_j\a_l \n \tilde{w}_j^i\cdot \n \tilde{w}_l^i}{1 - Cr_i^{n+2}}. 
\end{align*}
Now, we have
\[
\int|\n \tilde{w}_j^i|^2 = \l_j(\W) + \int |\n \tilde{w}_j^i|^2 - |\n u_j|^2.
\]
We thus have the right-hand side above is at most $Cr^n_i$ away from
\[
\sum_{j=1}^k \a_j^2 \l_j(\W).
\]
Using property S, we see that for the maximal $\a_j$, we must have $\a_k^2 \geq 1 - C r^n_i$, and $\a_j^2 \leq C r^n_i$ for $j<k$ (the constant depends on $\l_k-\l_j)$. Going back to the original form, we now have (seeing as how only one term has no copies of $\a_j$, $j<k$) that
\[
\l_k(\tilde{\W}^i)\leq \l_k (\W) + \int |\n \tilde{w}_k^i|^2 - |\n u_k|^2 + C(r_i^{n+2} + r_i^{3n/2}).
\]
We learn, then, that
\begin{align*}
F(\W) +\int_{\W}f + \int_{\W^c}g &\leq F(\tilde{\W}^i) \int_{\tilde{\W}^i}f + \int_{\tilde{\W}^i}g \\
&\leq F(\W) + \sum_{k=1}^N \xi_k \int |\n \tilde{w}_k^i|^2 - |\n u_k|^2 + [f(0)-g(0)][|\tilde{\W}^i|-|\W|] + o(r_i^n).
\end{align*}
Here $\xi_k$ are the partial derivatives as defined in the statement of the lemma. Rescaling and passing to the limit (notice that all terms pass to the limit strongly) gives the conclusion.
\end{proof}

\begin{remark}\label{rem:nonS}
	Let us examine what would have happened in (3) if we did not assume property (S). For each $\l_k$, let $E_k$ be the space of corresponding eigenfunctions for $\W$, and $m(k)$ the smallest index for which $\l_{m(k)}=\l_k$. In fact, the computation would still be the same, except we would find that
	\[
	\l_k(\tilde{\W}^i)\leq \l_k (\W) + \min_{A\ss E_k, k = \dim A - 1 + m(k)} \max_{u_*\in A} \int |\n \tilde{w}^i[u_*]|^2 - |\n u_*|^2 + C(r_i^{n+2} + r_i^{3n/2}),
	\]
	where if $u = \sum_j \b_j u_j$, then $w^i[u]= \sum_j \b_j w_j^i$. In the limit, we would then get that $\{v_k\}$ (we abuse the notation, saying that $v\in A\ss E_k$ if it is given by $v= \sum_j \b_j v_j$ and $u=\sum_j \b_j u_j\in A$) have the property that
	\begin{align*}
	\sum_{E_k  \text{distinct}}& \sum_{j=1}^{\dim E_k} \min_{A\ss E_k, \dim A=j}\max_{v\in A}\xi_{j+m(k)-1} \int |\n v'[v]|^2 - |\n v|^2\\
	&  \geq  \xi_0[|\{v_1>0\}\cap B_R|-|E\cap B_R|]  
	\end{align*}
	for all competitors $\{v'_k\}$ with $v_k'-v_k$ compactly supported on $B_R$. Another way of writing the term on the left is
	\[
	\sum_{E_k  \text{distinct}} \sum_{j=1}^{\dim E_k} \xi_{j+m(k)-1} \mu_j( \{\int \n v'_s\cdot \n v'_t - \n v_s \n v_t\}_{t,s=m(k),\ldots,m(k)+\dim E_k -1}),
	\]
	where $\mu_j$ is the $j-th$ smallest eigenvalue of the $\dim E_k$ by $\dim E_k$ matrix in the parentheses. In all of the above, $\xi_j=\p_{\l_k} F(\l_1,\cdots, \l_N)$, as in the last part of the lemma's statement (they are unrelated to constants from Property E, which we do not assume here).
\end{remark}

\begin{lemma}\label{lem:StoE} Property S implies property E with the constants $\xi_k$ given by $\p_{\l_k} F(\l_1,\cdots, \l_N)$.
 
\end{lemma}

\begin{proof} 
We show that for every $x\in \p^* \W$, we have that the limit
\[
\lim_{r\searrow 0} \sum_{k=1}^N\xi_k \1\frac{1}{\cH^{n-1}(\p \W\cap B_r(x))}\int_{B_r(x)\cap \p \W} (u_k)_\nu (x) d\cH^{n-1}\2^2
\]
exists and equals $\xi_0(x)$. Once this is established, the conclusion follows from the Radon-Nikodym theorem.

Set $x=0$. Take any sequence $r_j\searrow 0$; then passing to a further subsequence we have that the blow-ups
\[
u_{k,j}(x):= \frac{u_k(r_j x)}{r_j}\rightarrow v_k(x)
\]
locally uniformly for Lipschitz functions $v_k$. Moreover, we have $\p \W/r_k \rightarrow \p \{v_1>0\}$ locally in the Hausdorff sense, which implies that $\{v_1>0\}=\{x_n<0\}$ in a suitable coordinate system. We also have (up to a further subsequence, using distributional convergence and the fact that the limit measures vanish on $\p B_1$)
\[
 \cH^{n-1}(B_1 \cap \W/r_k)\rightarrow \cH^{n-1}(B_1\cap \{x_n=0\})
\]
and
\[
 \triangle u_{k,j}(B_1) + r_k^2 \l_k \int_{B_1} u_{k,j} \rightarrow  \triangle v_k(B_1).
\]
The odd reflection of each of the $v_k$ across $\{x_n=0\}$ is an entire harmonic function with bounded gradient; it follows that it is affine and of the form $v_k = \a_k x_n^-$. From the minimality property of Lemma \ref{lem:blowup}, it is simple to check (using domain variations) that
\[
 \sum_{k=1}^N \xi_k \a_k^2 = \xi_0(0).
\]
It follows that
\[
 \sum_{k=1}^N \xi_k\1\triangle u_{k,j}(B_1) + r_k^2 \l_k \int_{B_1} u_{k,j}\2^2 \rightarrow \sum_{k=1}^N \xi_k \int_{\{x_n=0\}\cap B_1} \a_k^2 = \xi_0(0)\cH^{n-1}(B_1\cap \{x_n=0\}),
\]
which gives 
 \[
  \lim_{r_j} \sum_{k=1}^N\xi_k \1\frac{1}{\cH^{n-1}(\p \W\cap B_{ r_j})}\int_{B_{ r_j}\cap \p \W} (u_k)_\nu d\cH^{n-1}\2^2= \xi_0(0)
 \]
for a subsequence of every sequence. It follows that the limit exists. 
\end{proof}

\begin{remark}\label{rem:nonS2} Continuing with Remark \ref{rem:nonS}, we examine if we might learn something even if property S is not known. Choose a point $x\in \p^* \W$ to perform blow-ups, and deduce that in this case too we have $v_k=\a_k x_n^-$. Take any $E_k$, and notice that we can find an orthonormal basis for $E_k$ such that $\a_{m(k)}^2 \geq 0$, and $\a_j=0$ for $j>m(k)$. Indeed, we have
	\begin{align*}
	T:&E_k \rightarrow \R\\
	& u = \sum_{j=m(k)}^{m(k)+\dim E_k -1} \b_j \mapsto \sum \b_j \a_j	 
	\end{align*}
	is a linear map, and has a kernel of dimension at least $\dim E_k -1$. Choose the maximum number of the new basis vectors in this kernel, and the final one orthonormal to them.
	
	Note that we are not assuming property E here, and the basis of eigenfunctions selected in this way will have nothing to do with the basis used to write property E in diagonal form. Indeed, unlike that basis, this one will depend on the point $x$ and (a priori, at least) the blow-up sequence being considered.
	
	Let us examine what happens if the image of this map $T$ is $\{0\}$ along every blow-up sequence at every $x\in \p^* \W$. Arguing as above, this implies that the eigenfunction $w=u_{m(k)}$ has the measure $\D w=0$ on $\p \W$; i.e. it satisfies the Neumann condition. This is a contradiction: for example, from integration by parts,
	\begin{align*}
	0&= \int_{\W} \l_k u_1 w - \l_1 u_1 w \\
	 &= \int_{\W} w \D u_1 - u_1 \D (w+1)\\
	 &= \int_{\p^*\W} (w+1) (u_1)_\nu - u_1 w_\nu d\cH^{n-1}\\ 
	 &= - \l_1 \int_{\W} u_1\neq 0.
	\end{align*}
	Thus at least at some $x\in \p^*\W$, there is a blow-up sequence with $\a_{m(k)}\neq 0$.
	
	Let us now consider a competitor where (on each distinct eigenspace $E_k$) only $v_{m(k)}$ is modified to $v'_{m(k)} = v_{m(k)}(x+t\phi(x)e_n)$ (with $\phi$ a nonnegative compactly supported smooth function ), and the others are left unchanged. The matrix in each term of the "minimization property" then takes the particularly simple form
	\[
	\int |\n v'_{m(k)}|^2 -|\n v_{m(k)}|^2 \d_{s,t=m(k)} = O(t^2) - \int_{\{x_n=0\}} \a_{m(k)}^2 t\phi \d_{s,t=m(k)}. 
	\]
	Then we see that
	\[
	-\sum_{E_k \text{distinct}} \xi_{m(k)} \int_{\{x_n=0\}} \a_{m(k)}^2 t\phi \geq - \xi_0 \int_{\{x_n=0\}} t \phi + O(t^2),
	\]
	and sending $t\searrow 0$shows
	\[
	\sum_{E_k \text{distinct}}\xi_{m(k)}\a_{m(k)}^2\leq \xi_0.
	\]
	On the other hand, using negative $t$ gives
	\[
	\sum_{E_k \text{distinct}} \xi_{m(k)+\dim E_k -1}\a_{m(k)}^2\geq \xi_0.
	\] 
	In particular, if $\xi_{m(k)} >\xi_{m(k)+\dim E_k -1}$ for every $E_k$ of dimension greater than one, and we are considering a blow-up sequence where $\a_{m(k)}\neq 0$ for at least one such $E_k$, we obtain a contradiction. 
	
	We have just shown the following: if $F$ has the property that
	\[
	\p_{\l_k} F(\l_1,\cdots, \l_N) > \p_{\l_{k+1}} F(\l_1,\cdots, \l_N)
	\] 
	whenever $\l_k=\l_{k+1}$, then minimizers for $F$ will always enjoy property S.
	\end{remark}

\begin{lemma}\label{lem:blowupreg} Let $\{v_k\}$ be a blow-up limit as above, and assume that either $\{v_1>0\}$ contains or is contained in $\{x\cdot \nu <0\}$. Then there exists a (possibly different) blow-up limit $\{\tilde{v}_k$ with $\tilde{v}_k = \a_k (x\cdot \nu)^-$ for constants $\a_k$ with $c\leq \sum_{i=1}^N \a_i^2\leq C$. Assume, moreover, that property E holds: then
\[
 \sum_{i=1}^N \xi_i \a_i^2 = \xi_0.
\]
\end{lemma}

\begin{proof}
 Denote by $\W_\8=\{v_1>0\}$ and $H=\{x\cdot\nu <0\}$. Then $v_1$ is a nonnegative, globally Lipschitz continuous harmonic function with $v_1=0$ on $\p \W_\8$, and satisfies $\max_{B_r(x)} v_1(x)\geq c r$ for any $x\in \W_\8$.
 
We now observe that this implies that $|B_r \cap (\W_\8 \triangle H)| = o(r^n)$. For $\W_\8\ss H$, this follows from Lemma 11.17 in \cite{CS} applied to $v_1$. In the case of $H\ss \W_\8$, we have from Remark 11.18 in \cite{CS} that $v_1 = \a x_n^- + o(|x|)$ on $H$. Performing a second blow-up at the origin, the function $v^*$ to which $v_1(r \cdot)/r$ converges along a subsequence has the same properties as $v$, and in addition vanishes on $\p H$. This, the linear growth estimate on $v_1$, and the positive density of $\{v_1 = 0\}$ implies that $v^* = \a x_n^-$ (see Lemma 12.8 in \cite{CS}, for example). This, in turn, gives that $v_1 = \a x_n^- + o(|x|)$ near $x$ (including outside of $H$). Together with the linear growth estimate imply $|B_r \cap (\W_\8 \sm H)| = o(r^n)$

Let $r_j$ be the sequence for which $v_j = \lim u_k(x r_j)/r_j:=u_{k,r_j}$. For each $\r<1$, we have that $|B_\r \cap (\{v_1>0\}\triangle \{u_{1,j}>0\}|$ goes to $0$. Thus for each $i$, we may find $j(i)$ such that $j(i)\geq j(i-1)$ and for all $j\geq j(i)$
\[
 |B_{2^{-l}} \cap (\{v_1>0\}\triangle \{u_{1,j}>0\}| \leq |B_{2^{-l}}| /i
\]
for each $l=1,\ldots,i$. Then the sequence $u_{k,r_{j(i)}/i}\rightarrow \tilde{v}_j$, which inherits the property that
\[
 |B_{2^{i}} \cap \{\tilde{v}_1>0\}\triangle H|=0.
\]
As the functions $\tilde{v}_k$ are harmonic on $H$ and vanish on $\p H$, their odd reflections are entire harmonic functions; combined with the Lipschitz bound and Liouville's theorem this implies the first conclusion.

The second conclusion is immediate from Lemma \ref{lem:blowup}.
\end{proof}

\begin{remark} A consequence of our main theorem is that if one blow-up is a half-plane, then all of the blow-ups are unique, even without property S. Hence a posteriori, we could take $\tilde{v}_k=v_k$ in the above.
\end{remark}

\begin{lemma}\label{lem:startcfg} Let $\W$ be a minimizer with $0\in \p \W$. Then for every $\e>0$, there are $\d,\r$ (depending on $\e$ and universal quantities) such that if $r<\d$ and
\[
 B_r \cap \p \W \ss \{|x\cdot \nu|<\d r\}\qquad  r \nu/2 \notin \W
\]
for a unit vector $\nu$, then there are numbers $\{\a_k\}$ with $c\leq \sum_{i=1}^N \a_i^2\leq C$ such that
\[
 \sup_{B_{\r r}}|\a_k (x\cdot \nu)_- -u_k|\leq \r r \e.
\]
\end{lemma}

\begin{proof}
 We argue by contradiction. Assume that the conclusion fails; then (for any $\r\in (0,1)$) there is a sequence $F_j$ of functionals and $\W_j$ of $F_j$-minimizers, such that $0\in \p \W_j$, along with $r_j,\d_j\searrow 0$ such that (choosing coordinates $x=(x',x_n)$ and performing a rotation to make $\nu_j = e_n$, and $(0,-r_j/2)\in \W_j$)
 \[
  B_{r_j} \cap \p \W_j \ss \{|x_n|<\d_j r_j\},
 \]
but
\[
 \sup_{B_{\r r_j}}\frac{|\a_k (x_n)^- -u_{k,j}|}{r_j\r}> \e.
\]
for any choice of constants $\a_k$ (it is clear from our universal upper and lower bounds on the eigenfunctions that this is trivially the case if the condition $c\leq \sum_{i=1}^N \a_k^2 \leq C$ is violated, for universal $c,C$).

Consider then the functions $v_{k,j}(x) = u_{k,j}(r_k x)/r_k$. These are defined on $B_2$, and have $|\n v_{k,j}|\leq C$ independently of $j$, while $v_{k,j}=0$ outside $\W_k/r_k$. Passing to a subsequence, we have that $v_{k,j}\rightarrow v_k$ uniformly on $\bar{B}_1$, and $v_k=0$ on $B_1 \cap \{x_n\geq 0$. Furthermore, as $\l_N(\W_k)$ is universally bounded, we have for $x\in B_1\cap \W_k/r_k$
\[
 |\triangle v_{k,j}(x)| = r_k |\triangle u_{k,j}(r_k x)|\leq C r_k^2.
\]
From this it follows that $v_k$ is harmonic on $B_1 \cap \{x_n<0\}$. Finally, we have that
\[
 \sup_{B_{\r}}\frac{|\a_k (x_n)^- -v_k|}{\r}\geq \e
\]
for every set of constants $\{\a_k\}$.

Now, let $\tilde{v}_k$ be the odd reflections of $v_k$:
\[
 \tilde{v}_k (x',x_n)=\begin{cases}
              v_k(x',x_n) &x_n\leq 0\\
              v_k(x',-x_n) &x_n>0.
             \end{cases}
\]
These are harmonic functions on $B_1$ which are bounded by a universal constant $C$, and so admit the estimate
\[
 |\tilde{v}_k(x) - x\cdot \n \tilde{v}_k|\leq C|x|^2
\]
for $x\in B_{1/2}$, with $\n \tilde{v_k} =\a_k e_n$ for some $\a_k$. If $\r <\e/C$, we have arrived at a contradiction.
\end{proof}

The next lemma, while requiring Property E, does not depend quantitatively on it.

\begin{lemma}\label{lem:goodcfg} Let $\W$ be as in Lemma \ref{lem:startcfg}. If property E holds, then we may take the values $\a_k$ to satisfy
\[
 \sum_{i=1}^N \xi_k \a_k^2 = f(0)-g(0)
 \]
 The values $\d,\r$ remain universal. 
\end{lemma}

\begin{proof}
Fix $\e>0$. Take a minimizer $\W$ as stated, with 
\[
 B_r \cap \p \W \ss \{|x\cdot \nu|<\d r\}
\]
for $r<\min\{\d,r_*\}$, with $\d=\d(\e_*)$ as in Lemma \ref{lem:startcfg} for $\e_*<\e/2$ and $\r_*$ depending only on $\e$ and universal quantities and to be chosen shortly. Now extract the values $\a_k$ given by that lemma, and assume that
\[
  \sum_{i=1}^N \xi_k \a_k^2 > \xi_0 +\e/2.
\]
We will now show this leads to a contradiction. The other inequality
\[
  \sum_{i=1}^N \xi_k \a_k^2 < \xi_0 -\e/2.
\] 
results in a contradiction by an analogous argument, and this will establish the 
conclusion.

From Lemma \ref{lem:startcfg}, we have that on $B_{\r r}$,
\[
 |u_k - \a_k (x\cdot \nu)^-| \leq \r r\e_*.
\]
First, we distinguish between indices $A = \{k: |\a_k|\leq \frac{\e}{2N}$ and $B = \{1,\ldots,N\}\sm A$. Note that for $\e$ small enough, $B$ is nonempty. After switching the sign of some of the $u_k$, we may assume that for every $k\in B$, $\a_k>\e/2N$. Let $w_k(x) = u_k(\r r x)/r\r$ be rescalings of the eigenfunctions, which satisfy
\[
 |\triangle w_k| \leq C (\r r)^2\leq C r_*^2
\]
on $\W\cap B_1$. We have
\begin{equation}\label{eq:ELcfg1}
 w_k(x) \geq \a_k (x\cdot \nu)^-  - \e_* \geq \a_k \1(x\cdot \nu)^-  - 2N\frac{\e_*}{\e}\2 
\end{equation}
on $B_1$ for each $k$.

Choose coordinates $(x=(x',x_n)$ such that $x\cdot \nu=x_n$. Fix a smooth function $\psi: \R^{n-1}\rightarrow [0,1]$ with $\psi(0)=1$ and $\psi$ supported on $B_{1/2}$, and $|\psi|+|\n \psi|+|\D^2\psi|\leq C_*$. Let $\phi_t$ be constructed as follows:
\[
 \phi_t(x) = (1 - \frac{20N\e_*}{\e}) (x_n - t\psi(x'))^- - 3N\frac{\e_*}{\e} - \frac{10 N C_*\e_*}{\e}(x_n - \frac{10N\e_*}{\e})^- (1-x_n) .
\]
Let us summarize the relevant properties of this family of functions. First, $w_k > \a_k \phi_0$ on $B_1$, as $\phi_0$ is strictly smaller than the right-hand side of \eqref{eq:ELcfg1}. Next, if $t\leq \frac{10 N \e_*}{\e}$ and $x_n<-\frac{1}{2}$ we have that
\[
 \phi_t(x)< (1 - \frac{20N\e_*}{\e}) [(x_n)^- + t] - 2N\frac{\e_*}{\e}
 \leq (x_n)^- -  \frac{20N\e_*}{\e}t - 2N\frac{\e_*}{\e},
\]
from which it follows that $w_k > \a_k \phi_t$ on $\p B_{1}$. Still for $t\leq \frac{10 N \e_*}{\e}$, we have that
\[
 \triangle \phi_t \geq -t C_* + 20N C_* \frac{\e_*}{\e} \geq 10 N C_*\frac{\e_*}{\e}
\]
on $x_n <\frac{10 N \e_*}{\e}$. For $x_n\geq t$, we easily see that $\a_k \phi_t < w_k$. When $t=\frac{10N\e_*}{\e}$, we have
\[
 \phi_t (0,2N\e/\e_*) = 5 N \e_*/\e - O(\e_*^2/\e^2)>0,
\]
so $\a_k \phi_t>w_k$ at that point. Finally, we have
\[
 \n \phi_t = e_n + O(\e_*/\e).
\]

Let $\bar{t}$ be
\[
 \bar{t}=\sup\{t: w_k >\a_k \phi_t \text{ on } B_1 \quad \forall k\in B\}.
\]
From the above observations, we have that $\bar{t}\in (0,\frac{10N\e_*}{\e})$. Take a point $x\in \{\a_k \phi_{\bar{t}} = w_k\}$ for some $k\in B$. By the discussion above, $x\notin \p B_1$. Moreover, provided we take $r_*^2 < C(\e_*/\e)^2$, $x$ cannot lie in $\{w_1>0\}$, for we would then have $\D (\a_k \phi_{\bar{t}} - w_k)>0$ at a local maximum. If $x\notin \p \{w_1 >0\}$, we also run into a contradiction, as then $w_k\equiv 0$ in a neighborhood of $x$, while $\n \phi_{\bar{t}}(x)\neq 0$. To summarize, we have discovered a point $x\in \p\{w_1>0\}$ such that on a small neighborhood of $x$, $w_k \geq \phi \a_k$ for a smooth function $\phi$ with $\p_{\nu_x} = 1 + O(\e_*/\e)$, where $\nu_x$ is the normal vector to $\{\phi=0\}$ at $x$.

Now perform a blow-up of $\{w_k\}$ at $x$, as per Lemma \ref{lem:blowup}. We have
\[
 \frac{w_k(x+s\cdot}{s}\rightarrow v_k
\]
along a subsequence, and $v_k(y)\geq \n \phi(x)\cdot y $ for $k\in B$. Applying Lemma \ref{lem:blowupreg} (and possibly changing sequences), we learn that
\[
 \xi_0 = \sum_{i=1}^N \xi_i |\n v_i|^2 \geq \sum_{i\in B} \xi_i \a_i^2 \geq \xi_0 +\e/2 - C\e_*/\e,
\]
which gives a contradiction if $C\e_*/\e^2$ is chosen small enough.
%
\end{proof}

\section{Verifying Property E} \label{sec:propE}

In this section we show that minimizers of $F$ satisfy property $E$, by some approximation arguments. The regularity arguments to follow will only rely on the following properties:

\begin{definition} For a natural number $N$, a constant $Q>1$, positive numbers $\{\xi_k\}_{k=1}^N$ with $|\xi_k|\leq Q$, and a Lipschitz function $\xi_0$  with $|\n \xi_0|\leq Q$ and $|\xi_0 -1|\leq Q^{-1}$, we say that an open set $\W\ss B_1$ with $0\in \p \W$ and a collection of functions $\{u_k\}_{k=1}^N$ are in $\cQ(N,Q,\{\xi_k\}_{k=0}^N)$ if the following hold:
\begin{enumerate}
 \item The functions $u_k$ satisfy $u_k=0$ outside $\W$ and $|\n u_k|\leq Q$.
 \item The functions $u_k$ satisfy
 \[
  \sup_{B_r(x),k}|u_k|\geq Q^{-1}r
 \]
for every $B_r(x)\ss B_1$ with $x\in \W$.
 \item On $\W$, we have $|\triangle u_k|\leq Q$.
 \item For every $x\in  B_1\cap \p\W$ for which there is a ball $B_r(y)\ss \W$ with $\{x\}=\p B_{r}(y)\cap \p \W$, and any $\b_k\geq 0$ for which
 \[
  |u_k(z)|\geq \b_k (\frac{y-x}{|x-y|}\cdot (z-x))^+ + o(|z-x|),
 \]
 we have that $\sum_{k=1}^N \xi_k \b_k^2 \leq \xi_0(x)$
 \item For every $x\in  B_1\cap \p\W$ for which there is a ball $B_r(y)\ss B_1\sm \W$ with $\{x\}=\p B_{r}(y)\cap \p \W$, and any $\b_k\geq 0$ for which
 \[
  |u_k(z)|\leq \b_k (\frac{y-x}{|x-y|}\cdot (z-x))^- + o(|z-x|),
 \]
 we have that $\sum_{k=1}^N \xi_k \b_k^2 \geq \xi_0(x)$.
\end{enumerate}
We say that $(\W,\{u_k\}_k=1^N)\in \cQ(N,Q)$ if there are $\{\xi_k\}_{k=0}^N$ with the above properties for which  $(\W,\{u_k\}_k=1^N)\in \cQ(N,Q,\{\xi_k\}_{k=0}^N)$.
\end{definition}

\begin{remark} Our results would be equally valid if $\xi_0$ was merely H\"older continuous, the coefficients $\xi_k$ depend in a H\"older manner on $x$, and doubtless under other generalizations. As we discuss below, however, the assumptions given already cover all of our desired applications.
\end{remark}

Note that we do not even assume that $u_k$ is an eigenfunction (nor do we assume any of the $u_k$ have a sign). It is elementary to check that for any element $(\W, \{u_k\})$ in $\cQ(N,Q)$, the dilates $(\W/r\cap B_1, \{u_k(r\cdot)/r\})\ss \cQ(N,Q)$. By itself, this set of hypotheses will not be particularly useful, however; we will usually wish to be in a \emph{flat configuration}:

\begin{definition} Let $(\W,\{u_k\})\ss \cQ(N,Q,\{\xi_k\})$.   Let the \emph{flatness} $f$ be given by 
\[
 f = \inf\{ \sup_{B_1\cap \W,k}| u_k + \a_k (x\cdot\nu)|:\nu\in S^{n-1}, \sum_{i=1}^k \xi_k \a_k^2 -\xi_0 (0)\}.
\]
\end{definition}

Provided $f<c(Q)$, $\nu/2 \in \W$, and $\nu/2\notin \W$, we have that for the optimal $\a_k$,
\[
 c Q^{-1}\leq \sum_{k=1}^N\a_k^2 \leq CQ.
\]
Such an element $(\W,\{u_k\})$ will be described as \emph{flat}.

Let $\W$ be a minimizer which satisfies property E. Then take any $x\in \p \W$ and $r$ small enough; we claim that $B_1\cap (\W-x)/r$ and the rescaled eigenfunctions are in $\cQ(N,Q,\{\xi_k\})$ with $Q$ universal. Indeed, the first three properties were established in Section \ref{sec:general}, while the last two are easy consequences of Lemma \ref{lem:blowupreg}. Moreover, if  $x\in \p^*\W$ we also have that the flatness goes to $0$ with $r$, by Lemma \ref{lem:goodcfg}, and so the configuration is flat.

Now, take any functional $F(\W) = G(\l_1(\W),\ldots,\l_N(\W)) + |\W|$  satisfying our original assumptions, and let $V$ be a minimizer of $F$. Set
\begin{align*}
 F_p(\W) &= G(\l_1, (\sum_{i=1}^2 \l_i^p)^{1/p},\ldots, (\sum_{i=1}^N \l_i^p)^{1/p}) \\
 &\qquad + \int_\W (1 + s \max\{d(x,V),1\}) +\int_{\W^c} s\max\{d(x,V^c),1\}.
\end{align*}

\begin{lemma} For $s>0$ small and fixed and each $p$ large enough, $F_p$ admits (at least one) bounded, open, connected minimizer $V_p$. Moreover, $V_p\rightarrow V$ in $L^1$ and the Hausdorff sense. If (for $k\leq N$) $v_k^p$ is an eigenfunction for $V_p$ with eigenvalue $\l_k(V_p)$, then $v_k^p \rightarrow v_k$ uniformly, with $v_k$ an eigenfunction for $V$ with eigenvalue $\l_k(V)$.
\end{lemma}

\begin{proof}
 Let 
 \[
F_\8(\W) = F(\W) +  \int_\W  s \max\{d(x,V),1\} +\int_{\W^c} s\max\{d(x,V^c),1\},  
 \]
and note that $V$ is the unique minimizer of $F_\8$ over open sets (it is a minimizer of the first term and a unique minimizer of the second). Select an $R$ large enough that $V\ss B_R$. Choose $R_0 = R+1$,  and observe that  $F_p$ satisfy our standard assumptions uniformly in $p$ (with $\eta=s$): indeed, the partial derivative with respect to $\l_k$ is given by 
 \[
 \sum_{j=k}^N \p_{\l_j} G (\l_1, (\sum_{i=1}^2 \l_i^p)^{1/p},\ldots, (\sum_{i=1}^N \l_i^p)^{1/p})(\sum_{i=1}^j \l_i^p)^{1/p-1} \l_k^{p-1}.
 \]
Each of the terms is bounded uniformly in $p$ and nonnegative, while the $j=k$ term is uniformly bounded from below. Fix $s<\eta_0$, so that Lemma \ref{lem:con2} applies. 

By the results of \cite{DB}, we may find a set $V_p^S$ which minimizes $F_p$ over all quasiopen sets contained in $B_{S}$, with any $S$. We wish to show that there is a sequence of such minimizers $V_p^S$ which remain in some fixed ball $B_{R_1}$ uniformly in $S$. If that were the case, we could pass to the limit in the $\g$ sense of \cite{DB}, and obtain a minimizer of $F_p$ over all quasiopen sets.

Note that if $\bar{u}^S= \sum_{k=1}^N |u_k^S|$, where $u_k^S$ are the eigenfunctions of $V_p^S$, then $|u^S|\leq C=C(\l_N(V_p^S))$ uniformly bounded and $\D \bar{u}^S \geq - \bar{u}^S 2\l_N(V_p^S)\geq -C$.  We fix $r<r_0$ as in Lemma \ref{lem:lowerbd}, and take any point $x$ with $|B_{r/8}(x)\cap V_p^S|>0$; then $\sup_{B_{r}(x)} |u_k^S|\geq c_0 r$, so there is some $y\in B_{r}(x)$ with $\bar{u}^S(y)\geq c_0 r$.  Using the mean value property and the bound on $\bar{u}^S$ gives
\[
c_0 r \leq \bar{u}^S(y) \leq \fint_{B_{r/2}(y)} \bar{u}^S  - C r^2 \leq C \frac{|V_p^S \cap B_{r/2(y)}|}{r^n} - Cr^2.
\]
Then if $r$ is small enough, the second term is reabsorbed, and we have
\[
|V_p^S \cap B_{r/2}(y)|\geq c r^{n+1}\geq c_1.
\]
It follows that there can only be $|V_p^S|/c_1$ (a number bounded uniformly in $S$) disjoint balls of the type described; this means that up to a set of $0$ Lebesgue measure, $V_p^S$ is contained in a uniformly bounded number $K$ of balls $\{B_T(x_i)\}_{i=1}^K$, with fixed $T$. Removing this Lebesgue-negligible portion of $V_p^S$ has no effect on the eigenvalues, so we may assume $V_p^S$ is truly contained in these balls. Finally, consider the connected components of $\cup_i B_T(x_i)$: if any of them lies outside of, say, $B_{R_0 + 2TK}$, we may translate it so that it at least intersects this ball while preserving the minimality (here we use that $F_p$ is translation-invariant outside of $B_{R_0}$).  This establishes the claim.

At this point, we have produced a minimizer $V_p$ for the functional $F_p$. From the results of Section \ref{sec:general}, it follows that $V_p$ is open, bounded, and connected, and that any basis of eigenfunctions $u_k^p$ satisfies $|\n u_k^p|\leq C$ as well as the nondegeneracy property of Lemma \ref{lem:lowerbd}. From this, along a subsequence we have $\l_k(V_p) \rightarrow \l_k$ for some numbers $\l_k$,  $u_k^p \rightarrow u_k$ uniformly for some Lipschitz functions $u_k$, and $V_p \rightarrow V_\8$ in $L^1$ and the Hausdorff sense. Here $u_k$ are supported on $V_\8$, and satisfy $-\triangle u_k = \l_k u_k$ on that set.

For any fixed open set $E$ we have $F_p(E)\rightarrow F_\8(E)$ as $p\rightarrow \8$, so this gives
\begin{align*}
 \inf_E F_\8 &\geq \lim_p \inf_E F_p(E)\\
 & = \lim_p F_p(V_p) \\
 &= F(\l_1,\ldots,\l_N) + \int_{V_\8}  s \max\{d(x,V),1\} +\int_{V_\8^c} s\max\{d(x,V^c),1\}.
\end{align*}
On the other hand, $\l_k \leq \l_k(V_\8)$ (as the $\l_k$ are eigenvalues, coming with corresponding eigenfunctions), so
\[
\inf_E F_\8 \leq F_\8(V_\8) \leq F(\l_1,\ldots,\l_N) + \int_{V_\8}  s \max\{d(x,V),1\} +\int_{V_\8^c} s\max\{d(x,V^c),1\}.
\] 
From this it follows that $V_\8$ is a minimizer of $F_\8$, and $\l_k = \l_k(V_\8)$. By the uniqueness of minimizers to $F_\8$, this means $V_\8 = V$. Finally, this argument may be applied to any subsequence, which gives the convergence claimed.
\end{proof}

\begin{lemma}\label{lem:Eforall} The following hold for the $V_p$ from above:
\begin{enumerate}
 \item $V_p$ enjoys property S.
 \item $V_p$ enjoys property E, with coefficients $\xi_{k,p}=\p_{\l_k}F_p(V_p)$ uniformly bounded in $p$.
 \item $V$ satisfies property E, with $\xi_0 =1$ and
 \[
 \xi_k = \lim_p \xi_{k,p}
 \]
 along a subsequence of $\{p\}$.
\end{enumerate}
\end{lemma}

To clarify the third conclusion, property E will hold with values $\{\xi_k\}$ which are any limit point of the collection $\{\xi_{k,p}\}_p$; we do not show that there is a unique limit point.

\begin{proof}
 To see that $V_p$ has property S, let us compute the $\l_k$ partial derivative of $F_p$. This is given by:
 \[
 \sum_{j=k}^N \p_{\l_j} G (\l_1, (\sum_{i=1}^2 \l_i^p)^{1/p},\ldots, (\sum_{i=1}^N \l_i^p)^{1/p})(\sum_{i=1}^j \l_i^p)^{1/p-1} \l_k^{p-1}.
 \]
Assume that for some $k$, $\l_k=\l_{k+1}$. Then we clearly have
\[
\p_{\l_k} F_p > \p_{\l_{k+1}}F_p,
\]
for the two expressions share $N-k$ identical terms, but then $\p_{\l_k} F_p$ has one additional strictly positive term. Applying Remark \ref{rem:nonS2}, we have that $V_p$ satisfies property S. By Lemma \ref{lem:StoE}, we also have (2).

To show (3), apply the criterion \eqref{eq:Edomvar} from Lemma \ref{lem:domvar}; this is satisfied uniformly in $p$ by each $V_p$, and the inequality passes to the limit using the dominated convergence theorem.
\end{proof}

\section{Harnack Inequality}\label{sec:harnack}

We now prove an estimate near flat points of the free boundary. Estimates of this type were introduced by \cite{DeSilva} for the Alt-Caffarelli one-phase problem, and have since been used in several related works \cite{DFS1,DFS2,DFS3,DFS4}.

\begin{lemma}\label{lem:harnack} There is a constant $\e_H=\e_H(N.Q)>0$ such that for any $\e<\e_H$, flat $(\W,\{u_k\})\in \cQ(N,Q)$ and nonnegative numbers $\{\a_k\}_{k=1}^N$ and $a\leq b$ satisfying
\begin{enumerate}
 \item $\sum_{i=1}^N \a_k^2 \xi_k = \xi_0$.
 \item Either (A) $\a_k [-x_n + a] <u_k< \a_k [-x_n + b]$ and $|\triangle u_k|\leq \e^2 \a_k$  in $\W$ or (B) $|u_k|\leq \e^2$ and $|\triangle u_k|\leq \e^4$ in $\W$ for each $k$.  
 \item $b-a\leq \e$.
 \item $\osc_{B_1} \xi_0 \leq \e^2$.
\end{enumerate}
we have that for some $a',b'$ with $a\leq a'\leq b'\leq b$, 
\[
\a_k [-x_n+a'] <u_k< \a_k [-x_n + b']
\]
on $B_{1/10}\cap \W$ for every $k$ for which alternative $(A)$ held, and moreover $b'-a'<(1-\g)\e$ with $\g\in (0,1)$ universal.
\end{lemma}

\begin{proof} 
We categorize the indices $k$ into either $\cA$ (for which alternative (A) holds) and $\cB$ (for which (B) holds). The set $\cA$ is nonempty, for otherwise we contradict the nondegeneracy property; indeed, $\a_k^2 \xi_k \geq Q^{-1}$ for at least one $k$ (say without loss of generality  for $k=1$). Note also that as $0\in \p \W$, we must have $a<0<b$, and in particular $|a|,|b|\leq \e$. Set $\e' = b-a$; we may as well assume that $\e'\geq \e/2$, for otherwise the conclusion holds automatically with $\g=1/2$. We have that for $k\in \cA$, $u_k \geq -\a_k \e'$.

We have that for every $k\in \cA$,
\[
 \a_k [-x_n+a] \leq u_k \leq \a_k [-x_n+a+\e']
\]
in $\W$. Set $z = (0,-1/5)$ (we use the notation $x=(x',x_n)$); let us first consider the case of
\[
 \a_1 (a+\e'/2-z_n) \leq u_1 (z).
\]
The opposite case can then be treated analogously. From the Harnack inequality applied to
\[
 \frac{u_1 - \a_k (a-x_n)}{\a_k \e'}, 
\]
we have that $u_1 > \a_k (a+c_0 \e + 1/5)$ on $\{(x,x'): x_n=-1/5, |x'|\leq 9/10\}$ for some universal $c_0$.  

We will require some barrier functions, which we construct now. First, let $q:[0,1]\rightarrow [0,1]$ be a smooth nonincreasing function with $q=1$ on $[0,1/2]$ and $q=0$ on $[3/4.1]$. Let $q_1 : [0,1]\rightarrow [-1.0]$ be a smooth nondecreasing function which is $-1$ on $[0,1/10]$ and $0$ on $[3/20,1]$. Let $Q_s$ be the region
\[
 Q_s = \{(x',x_n): |x'|\leq \frac{9}{10}, -1/5<x_n< b + [q(|x'|)+\k q_1(9/10-|x'|)] s\}.
\]
for a $\k$ universal to be chosen shortly. We write $\p^f Q_s$ for the part of $\p Q_s$ where $x_n =[q(|x'|)+\k q_1(9/10-|x'|)] s$.  Note that $Q_s\ss B_1$ provided $s<\e$. Now set $\phi_s$ to be the solution to the Dirichlet problem
\[
 \begin{cases}
  \triangle \phi_s = \e^2 & x\in Q_s\\
  \phi_s = -\e'  & x\in \p^f Q_s\\
  \phi_s = - x_n + a + s \k q_1(-x_n) \in \p Q_s\sm \p^f Q_s.
 \end{cases}
\]
Consider the function $v_s = \phi_s - a + x_n$. We claim that for $\k$ suitably small, we have that $v_s \geq c_1 s -C_0 \e^2$ on $B_{1/10}\cap Q_s$, for all $s$. Indeed,set $v_s^\pm$ to be the solutions to
\[
 \begin{cases}
  \triangle v^+_s = \e^2 & x\in Q_s\\
  v_s^+ =  (x_n - b)^+ & x\in \p Q_s
 \end{cases}
\]
and
\[
 \begin{cases}
  \triangle v^-_s = 0 & x\in Q_s\\
  v_s^- = -(x_n - b)^-  & x\in \p^f Q_s\\
  v_s^- = s \k q_1(-x_n) & x\in \p Q_s \sm \p^f Q_s.
 \end{cases}
\]
(The support of $v_s^+$ on $\p^f Q_s$ is near the center, where $|x'|\leq 1/2$, while the support of $v_s^-$ is near the sides, where $|x'|\geq 6/10$). Then $v_s = v_s^+ +v_s^-$. Now, for $v_s^+$,  we have $v_s^+ \geq c s -C\e^2$ on $B_{1/10}\cap Q_s$. Indeed, for a harmonic function with the same boundary data as $v_s^+$, we would have that it is nonnegative and at least $s$ on a piece of the boundary $\p^f Q_s\cap B_{1/2}$. It follows that this harmonic function is bounded from below by $c s$ on $Q_s \cap B_{1/10}$; meanwhile, the solution the the Poisson problem 
\[
\begin{cases}
\triangle \theta_s = \e^2 & x\in Q_s\\
\theta_s =  0 & x\in \p Q_s
\end{cases}
\]
is bounded by $C\e^2$. On the other hand, clearly $|v_s^-|\leq \k s$ by the maximum principle, and so selecting $\k$ small enough, our original claim is valid. 

Once we have chosen $\k$, we check that for any point $y$ with $|y'|=9/10$ and $y \in \p^f Q_s$, $v_s$ is negative on $B_\r(y)$ for $\r$ universal. Indeed, on $B_{1/10}(y)\cap \p Q_s$, we have that $v_s = - \k s$, while on $\p B_{1/10}(y) \cap Q_s$, $v_s \leq s$ by the maximum principle. Another application of the maximum principle will then reveal that $v_s$ is less than an affine function which takes the value $-\k s$ on $y$ and is at least $s$ on $\p B_{1/10}\cap Q_s$. We may find such an affine function whose negativity set contains $B_\r(y)$, uniformly in $s$.

Also, from standard estimates on $v_s$, we have that $|\n \phi_s + e_n|\leq C(s+\e^2)$, with the constant depending only on $q,q_1$ and hence universal. From this we see that $\phi_s$ is strictly decreasing in $x_n$, and $\{\phi_s=0\}$ can be expressed as a smooth graph $(x', \t_s (x'))$. From the above argument, we know that $\t\leq a$ for $|x'|>9/10-\r$. Moreover, it follows that the normal derivative $(\phi_s)_\nu$, $\nu = -\n u/|\n u|$ along this graph satisfies
\[
 |(\phi_s)_\nu + 1|\leq C_1(s + \e^2).
\]

Next, let $\z_s$ solve the Dirichlet problem
\[
 \begin{cases}
  \triangle \z_s = 0 & x\in Q_s\cap \{\phi_s>0\}\\
  \z_s = 0 & x\in \p (\{\phi_s>0\})\sm \{x_n=-1/5\}\\
  \z_s = c_0 \e q(|x'|) & x\in \p Q_s\cap \{x_n=-1/5\}.
 \end{cases}
\]
It is straightforward to check using the Hopf lemma that for all $s\in [0,\e)$, we have
\[
 (z_s)_\nu \leq - c_2 \e
\]
for points with $|x'|\leq 9/10 -\r$ on $\{\phi_s=0\}$. We extend $z_s$ by $0$ to $Q_s$; note that this makes it subharmonic on $Q_s$.

Armed with these barriers, we now establish our main claim. By definition, we have that $u_k > \a_k (-x_n + a)> \a_k \phi_0$ for each $k\in \cA$ on $Q_0$. Furthermore, from the maximum principle we have $u_1 > \a_1 (\phi_0 +\z_0)$. Let $\bar{s}$ be the smallest $s$ for which either $u_k = \a_k \phi_s$ for some $k$ or $u_1 = \a_1 (\phi_s +\z_s)$ on $Q_s$, and assume $\bar{s}<\e$ (if not, proceed to the next step). Let $y$ be a point at which one of the equalities above hold. Then $y\notin \p Q_{\bar{s}}$: on $\p Q_{\bar{s}} \sm \p^f Q_{\bar{s}}$ we have $\phi_{\bar{s}}\leq \phi_0$ and $\z_{\bar{s}}=\z_0$, and the inequality was strict at $s=0$. Meanwhile on $\p^f Q_{\bar{s}}$ $\phi_{\bar{s}} = -e' $, and $u_k > \a_k \e'$.

We may also dispense with the case of $y\in \W$, for $|\triangle u_k|< \e^2 \a_k$, while $\triangle \a_k \phi_{\bar{s}} = \e^2 \a_k$ and $\triangle z_{\bar{s}} \geq 0 $. The same argument applies if $y \notin \p \W$. As $y\in \p \W$, we have $u_k(y)=\phi_{\bar{s}}(y)=\z_{\bar{s}}(y)=0$, so there is equality in \emph{every one} of the inequalities. From the properties of the barrier, this implies that $|y'|\leq 1-\r$. Set $\nu = -\n \phi_{\bar{s}}/|\n \phi_{\bar{s}}|$. Now, for each $k\in \cA$, we have
\[
 u_k(z)\geq \a_k \phi_{\bar{s}}(z) \geq \a_k (y-z)\cdot \nu (1-C_1 (\bar{s}+\e^2)) + o(|z-y|) 
\]
on any non-tangential cone contained in $\{\phi_k>0\}$. For $k=1$, we have the stronger inequality
\[
 u_1(z)\geq \a_1 (\phi_{\bar{s}}(z) +\z_{\bar{s}}(z)) \geq \a_1 (y-z)\cdot \nu (1-C_1 (\bar{s}+\e^2) + c_2 \e) + o(|z-y|).
\]
Applying the definition of $\cQ$ (with $\b_k=0$ for $k\in \cB$), we learn that
\[
2\a_1^2 \xi_1 (c_2 \e)(1-C_1(\bar{s}+\e^2)) + \sum_{k\in \cA} \xi_k \a_k^2 (1-C_1 (\bar{s}+\e^2))^2 \leq \xi_0(y). 
\]
Reorganizing and using assumptions (1) and (4),
\[
CQ^{-1} (c_2 \e)(1-C_1(\bar{s}+\e^2)) +  [\xi_0(0)-C\e^4](1-C_1 (\bar{s}+\e^2))^2 \leq \xi_0(0) + \e^2.
\]
Reabsorbing terms higher order in $\e$,
\[
 c_2 \e \leq C (C_1+1) (\bar{s}+\e^2).
\]
We see that provided $\e_H$ is small enough, this gives $\bar{s}\geq c_3 \e$.

Recall that $\phi_s \geq - x_n + a + c_1 s - C_0\e^2$ on $B_{1/10}\cap Q_s$. It follows that $u_k \geq \a_k ( -x_n + a + c_4 \e)$ on this set for each $k\in \cA$. Now, $Q_s$ contains $\W\cap B_{1/10}\cap \{x_n<b\}\cap B_{1/10}$ so we have our conclusion with $a'=a +c_4 \e$ and $b'=b$.

Let us briefly explain how to proceed if the opposite inequality holds,
\[
 \a_1 (a+\e'/2-z_n) > u_1 (z).
\]
The argument is essentially the same, using analogous barrier constructions. Once a contact point $y$ is identified, a comparison argument can be used to show that $u_k(z)\leq C\e^2 d(y,z)$ on a non-tangential region near $y$ for $k\in \cB$, so this contributes negligibly to the free boundary condition. A similar argument excludes the possibility that $(u_k)_\nu$ is large and positive for some $k\in \cA$.
\end{proof}

The purpose of the previous lemma is the following estimate:

\begin{corollary}\label{cor:holderest} There are universal constants $C_1,\e_1,\s>0$ such that the following holds: if  $(\W,\{u_k\})\in \cQ(N,Q)$ is flat and there are nonnegative numbers $\{\a_k\}_{k=1}^N$ and $a\leq b$ satisfying for $\e<\e_1$
\begin{enumerate}
 \item $\sum_{i=1}^N \a_k^2 \xi_k = \xi_0$.
 \item Either (A) $\a_k [-x_n + a] <u_k< \a_k [-x_n +b]$  and $|\triangle u_k|\leq \e^2 \a_k$ in $\W$ or (B) $|u_k|\leq \e^2$ and $|\triangle u_k|\leq \e^4$ in $\W$ for each $k$.  
 \item $b-a\leq \e$.
 \item $\osc_{B_1} \xi_0 \leq \e^2$,
\end{enumerate}
then for any $x,y\in B_{1/2}\cap \bar{\W}$ with $|x-y|\geq \e/\e_1$ and $k$ such that (A) holds, we have
\[
 |v_k(x)-v_k(y)|\leq C_1 |x-y|^\s,
\]
where
\[
 v_k (z) = \frac{u_k(z) + \a_k z_n}{\a_k \e}.
\]
\end{corollary}

\begin{proof}
 Assume first that $x\in \p \W$, and set $\e_1= \e_H/2$. Then applying Lemma \ref{lem:harnack} iteratively, we obtain that on $B_{10^{-i}/2}(x)$,
\[
 \a_k [=(z-x)_n + a_i]\leq u_k \leq \a_k [-(z-x)_n + b_i]
\]
with $b_i-a_i\leq 2(1-\g)^i \e$, provided that $2\cdot 10^i \e \leq \e_H$. We may rewrite
\[
 v_k (x) - v_k (y) = \frac{\a_k [x_n - y_n] -u_k(y) }{\a_k \e},
\]
and from the above inequality with $i$ such that $10^{-i-1}/2 \leq |x-y|\leq 10^{-i}/2$, we learn that
\[
 |v_k(x)-v_k(y)|\leq \frac{\a_k (b_i-a_i)}{\a_k \e}\leq 2 (1-\g)^{i} \leq C |x-y|^\s.
\]
Note that the assumption that $|x-y|\geq \e/\e_1$ ensures that $2\cdot 10^i \e \leq \e_H$.

If neither of $x,y$ are in $\p \W$, but either $d(x,\p \W)$ or $d(y, \p \W)$ is less than $2 |x-y|$ (say that $d(x, \p \W)\leq d(y, \p \W)$), find a $z\in \p \W$ with $|x-z|\leq 4 |x-y|$ and yet $\min\{|x-z|,|y-z|\} \geq |x-y|$. To see that such a $z$ exists, consider the annulus $W=\{w:|x-y|\leq |x-w|\leq 4 |x-y|\}$. From the assumptions, the free boundary is contained in the strip $\{a<x_n<b\}$ of thickness at most $\e$, and projects surjectively onto the plane $\{x_n=0\}$. On the other hand, $|x-y|\geq \e/\e_1 \gg \e$, and so (provided $\e_1$ is small) $W\cap \p \W$ is a set of diameter at least $4|x-y|$. Within this set we choose any point $z$ outside of $B_{|x-y|}(y)$. Then from the triangle inequality,
\[
 |v_k(x) - v_k(y)|\leq |v_k(x) - v_k(z)|+|v_k(z) - v_k(y)|\leq C|x-y|^\s.
\]

Finally, we have the case that both $d(x,\p \W)$ and $d(y,\p \W)$ are at least $2|x-y|$. Select $r=d(x,\p \W)$; then we have that
\[
 \osc_{B_r(x)} v_k \leq Cr^\s.
\]
Moreover, $v_k$ satisfies $|\triangle v_k|\leq \e$ on $B_r(x)$, so a standard estimate for the Laplace equation gives
\[
 |v_k(x)-v_k(y)|\leq C \frac{|x-y|}{r}[r^\s +r^2 \e] \leq C|x-y|^\s.
\]
\end{proof}

\section{Flatness Improvement}\label{sec:flat}

We now show that if all of the $u_k$ are either close to a nontrivial one-dimensional profile or extremely small, then on a smaller ball, those $u_k$ in the first category are even closer to some (slightly modified) profile. In the next section, we will deal with the propagation of smallness for the other $u_k$.

\begin{theorem}\label{thm:flatimp} There is a universal constant $C_2$ such that for every $\r<\r_0$ there is an $\e_F(\r)>0$ for which the following holds: if  $(\W,\{u_k\})\in \cQ(N,Q)$ is flat and there are nonnegative numbers $\{\a_k\}_{k=1}^N$ satisfying for $\e<\e_F$
\begin{enumerate}
 \item $\sum_{i=1}^N \a_k^2 \xi_k = \xi_0(0)$.
 \item Either (A) $\a_k [-x_n -\e] <u_k< \a_k [-x_n + \e]$  and $|\triangle u_k|\leq \e^2 \a_k$ in $\W$ or (B) $|u_k|\leq \e^2$ and $|\triangle u_k|\leq \e^4$  in $\W$ for each $k$.
 \item $\osc_{B_1} \xi_0 \leq \e^2$,
\end{enumerate}
then there are numbers $\{\a_k'\}$ also satisfying (1) and a direction $\nu$ such that for each $k$ for which (A) holds,
\[
 \a'_k [- x\cdot \nu -\frac{\r\e}{2}] <u_k< \a'_k [-x \cdot \nu + \frac{\r\e}{2}] \qquad \text{ in } \W \cap B_\r,
\]
and $|\nu -e_n|\leq C_2 \e$, $|\a_k - \a'_k|\leq C_2 \a_k \e$.
\end{theorem}

\begin{proof}
\textbf{I: Compactness}
 We argue by contradiction. If the conclusion fails for some $\r<\r_0$, we may find a sequence $(\W^i,\{u_k^i\})$ and $\e^i\searrow 0$ satisfying (1)-(3) but not the conclusion. Up to passing to a subsequence and rearranging, we may assume that for all $i$, $k=1,\ldots,M$ satisfy (A) and $k=M+1,\ldots, N$ satisfy (B), with $1\leq M\leq N$. Again passing to a subsequence, we may further assume that $\a_k^i \rightarrow \a_k\geq 0$ for each $k$, with $\a_k=0$ for $k>M$. We may also assume that $\xi_k^i \rightarrow \xi_k$ constant, and from (1)
\[
 \sum_{i=1}^N \a_k^2 \xi_k = \sum_{i=1}^M \a_k^2 \xi_k= \xi_0.
\]
In particular, $\a_k^2\xi_k\geq Q^{-1}$ for some $k$, say $k=1$. 

For each $k=1,\ldots,M$, set $v_k^i$ to be
\[
 v_k^i(x) = \frac{u_k^i(x) + \a_k^i x_n }{\a_k^i \e^i}.
\]
From (2), we know that $v_k^i \in (-1,1)$ on $\bar{\W}^i$. We also have $\bar{\W}^i\rightarrow \bar{B}_1^-= \{x\in B_1: x_n \leq 0\}$ in Hausdorff distance, and the graph of $v_k^{i}$ over $\bar{\W}^i$ converges to a relatively closed set $K_k\ss \bar{B}_1^- \times [-1,1]$ in Hausdorff distance. Applying Corollary \ref{cor:holderest}, we see that for $i$ large,
\[
 |v_k^i(x) - v_k^i(y)|\leq C|x-y|^\s
\]
for all $x,y\in B_{1/2}\cap \bar{\W}^i$ with $|x-y|\geq \e_i/\e_1$. Passing to the limit, it follows that $K_k$ is in fact the graph of a H\"older continuous function $v_k$ on $\bar{B}^-_{1/2}$. Furthermore, we have that
\[
 |\triangle v_k^i|\leq \e^i
\]
on $\W^i$, so $v_k^i \rightarrow v_k$ locally uniformly on the open set $B_{1/2}^-$, and the $v_k$ are harmonic there. Note also that as $0\in \p \W^i$ for each $i$, we have $v_k^i(0)=v_k(0)=0$.

We now study the $v_k$ near the boundary $H=\{x\in B_{1/2}: x_n=0\}$. First of all, take any $x\in H$ and $x^i\in \p \W$ with $x^i \rightarrow x$. we have that $v_k^i(x^i)\rightarrow v_k(x)$ from the discussion above. As $u_k^i(x^i)=0$, though, we see that $v_k^i(x^i)=v_1^i(x^i)$ for each $k$. We thus obtain the $M-1$ distinct relations
\[
 v_k = v_1 \qquad \text{ on } H.
\]
We now seek one additional boundary condition on $H$ which will imply that the $v_k$ solve a well-posed boundary value problem. Indeed, set $h=\xi_0^{-1}\sum_{i=1}^M \a_k^2 \xi_k v_k$; we claim that $h$ satisfies the Neumann condition $\p_n h =0 $ on $H$. We will check this condition in the viscosity sense: for each $z\in H$, $\r$ sufficiently small, and $\phi(x) =a + q x_n + \s_1 x_n^2 + \s_2 [(n-1)x_n^2 - |x'-z'|^2]  $ (with $\s_1,\s_2>0$), $a=\phi(z)=h(z)$, and $\phi<h$ on $\bar{B}_\r^-(z)\sm \{z\}$, we must establish that $q \geq 0$ (and also the analogous property for functions touching $h$ from above, which will follow similarly).

\textbf{II: Neumann Condition}
Let $\phi$, $z$, and $\r$ be fixed as above, and yet $q<0$. We begin by studying the functions $v_k-h$, which are harmonic on $B_1^-$ and vanish on $H$. This means they are $C^2$ on the closed half-ball, and in particular admit the Taylor expansion
\[
 (v_k-h)(y) = q_k y_n + O(|y-z|^2)
\]
on $B_\r^+(z)$. Then, up to replacing $q$ by $q/2$, increasing $\s_2$, and decreasing $\r$, we have that for every $k$, $v_k \geq q_k y_n + \phi$ on $\bar{B}_\r^-(z)$ with equality only at $z$.

Next, let $\w^i$ bound the distance between the graphs of $v_k^i$ and $v_k$ (noting that $\w^i\rightarrow 0$). In particular, this means that
\[
 v_k^i \geq q_k y_n + \phi - \w^i
\]
for each $k\leq M$ on $\W^i\cap B_\r^-(z)$. For a constant $C$ depending on $\phi$, we still have
\[
 v_k^i \geq q_k y_n + \phi - C\w^i
\]
on $\W^i\cap B_\r(z)$. Rewriting, this implies
\[
 u_k^i(y) \geq \a_k^i(-1 + q_k \e^i)y_n + \a_k^i \e^i \phi -C\a_k^i \e^i \w^i.
\]
We then also have
\[
 u_k^i(y) > \a_k^i(1 - q_k \e^i)[-y_n + \e^i \phi - C_0 \e^i (\w^i +\e_i)]
\]
(where $C_0$ is a larger constant depending on $\phi$).

Set $\phi^i_t$ to be
\[
 \phi^i_t = -y_n + \e^i \phi + t \e^i (\w^i +\e_i).
\]
For $t=-C_0$, we have that $u_k^i > \a_k^i(1 - q_k \e^i)\phi^i_t$ on $B_\r(z)\cap \W^i$. However, when $t>2$ and $\e^i q_k\ll 1$, we have (from the definition of $\w^i$) that $u_k(z) < \a_k^i(1 - q_k \e^i)\phi^i_t$. Let $t_*$ be the largest $t$ for which (for all $k$)
\[
 u_k^i \geq \a_k^i(1 - q_k \e^i)\phi^i_t \qquad k=1,\ldots,M_1; y\in B_\r(z), 
\]
and assume that we have equality at a point $x\in \bar{B}_\r(z)$ for some index $k$. Let us check that $x$ must lie in the interior of $B_\r(z)$, at least for large $i$. Indeed, we have that $v_k \geq \eta + \phi + q_k y_k$ for some $\eta>0$, on $\p B_\r^- (z)$. Hence on $\p B_\r(z)\cap \bar{\W}^i$, we have 
\[
 u_k^i \geq \a_k^i(1 - q_k \e^i)[-y_n + \e^i \phi -C_0 \e^i (\w^i +\e_i) +\e^i \eta].
\]
For large $i$, $\eta > (t+C_0)(\w^i+\e^i)$ for the entire range of $t$ considered, and so
\[
 u_k^i > \a_k^i(1 - q_k \e^i)[-y_n + \e^i \phi -t \e^i (\w^i +\e_i)]
\]
on $\p B_\r(z)$. Hence $x\in B_\r(z)$. Furthermore, we have that 
\[
\triangle \a_k^i(1 - q_k \e^i) \phi_t^i \geq \frac{1}{2} \a_k^i \s_1 \e^i 
\]
while by assumption 
\[
 |\triangle u_k^i| \leq  \a_k^i (\e^i)^2.
\]
This means $x$ cannot lie in $\W^i$ (or in the complement of $\bar{\W}^i$), and so must be in the free boundary $\p \W^i$. Then notice that we have equality at $x$ for \emph{every} $k=1,\ldots,M$, as both sides are $0$ there. The normal derivative $\p_\nu \phi_t^i(x) = - |\n \phi_t^i(x)|$ in the direction $\nu = -\n \phi_t^i(x)/|\n \phi_t^i(x)|$ has $|\p_\nu \phi_t^i(x) - (-1 + \e^i q)|\leq C\e^i \r$ (here $C$ depends on $\s_1,\s_2$). Up to selecting $\r$ smaller in terms of $\phi$, this gives $\p_\nu \phi_t^i(x) \leq -1 + \e^i/2 q$ (recall that we are assuming for contradiction that $q<0$).

Applying the definition of $\cQ$ with $\b_i = \a_k^i(1 -q_k \e^i)|-1 +\e^i q/2|  $ for $k\leq M$ and $\b_k=0$ otherwise, we see that
\[
 \sum_k=1^{M} \xi_k^i (\a_k^i)^2[1  -2 \e^i q_k -\e^i q ]\leq \xi_0^i + C (\e^i)^2; 
\]
all terms with $(\e^i)^2$ were moved to the right and combined. Using the definition of $M$, this is equivalent to
\[
 \sum_k=1^{M} \xi_k^i(\a_k^i)^2 [1 -2 \e^i  q_k -\e^i  q ] + \sum_{k=M+1}^N \xi_k^i (\a_k^i)^2 \leq \xi_0^i + C (\e^i)^2, 
\]
and so
\[
 \sum_k=1^{M} \xi_k^i (\a_k^i)^2 [ -2  q_k - q ] \leq C \e^i.
\]
Now take the limit in $i$, to get
\[
 \sum_k=1^{M} \xi_k \a_k^2 [ -2 q_k -q ] \leq 0.
\]
From the definition of $h$ and $q_k$, we easily see that
\[
 \sum_{k=1}^M \xi_k \a_k^2 q_k = 0,
\]
which implies that $q\geq 0$, contradicting our assumption. The other viscosity inequality may be verified similarly.

\textbf{III: Conclusion}
We have just shown that $h$ satisfies the Neumann condition on $H$ in the viscosity sense. In addition, each $v_k$ has $|v_k|\leq 1$ on $B_{1/2}^-$, and so $|h|\leq 1$ as well. It follows that $h$ is actually a classical solution (see \cite{DeSilva}), and that $h$ is $C^2$ on $\bar{B}_{3/8}^-$ with an estimate
\[
 \|h\|_{C^2(\bar{B}_{3/8}^-)} \leq C=C(n).
\]
Recalling that $v_k -h$ is harmonic, bounded by $2$, and vanishes on $H$, we have the standard estimate for the Dirichlet problem:
\[
 \|v_k\|_{C^2(\bar{B}_{1/4}^-)} \leq C=C(n).
\]
In particular, it follows that $v_k$ admits the Taylor expansion
\[
 v_k(x) = v_k(0) + \n v_k(0)\cdot x + O(|x|^2)= p_k x_n + \t \cdot x' + O(|x|^2),
\]
(recalling that $v_k(0)=0$). The lack of subscript on $\t$ is very much intentional: as $v_k -v_1 =0$ on $H$, we have that all of the tangential derivatives of the $v_k$ are identical. Using the Neumann condition on $h$, we see that
\[
 \sum_{k=1}^M \xi_k \a_k^2 p_k = \p_n h = 0.
\]

Let $\w^i$ bound the Hausdorff distance between the graphs of $v_k^i$ and $v_k$ for all $k=1,\ldots,M$, as well as $\max_k |\xi_k\a^2_k - \xi^i_k(\a_k^i)^2|$; we have that $\w^i\rightarrow 0$. Then on $\W^i \cap B_{1/4}^-$, we have that
\[
 |v_k^i - v_k|\leq \w^i,
\]
and so
\[
 |v_k^i(x) - p_k x_n - \t \cdot x'|\leq \w^i + C|x|^2.
\]
This implies that
\[
 |v_k^i(x) - p_k x_n - \t \cdot x'|\leq C(\w^i + |x|^2)
\]
on $B_{1/4}\cap \W^i$: indeed, for any $x\in B_{1/4}^+\cap \W^i$, we have that
\[
 \inf_{y\in B_{1/4}^-} |v_k^i(x) - v_k(y)| + |x-y|\leq C\w^i,
\]
while
\[
 |v_k(y) - p_k x_n -\t \cdot x'|\leq C|x-y| +C|y|^2.
\]
Combining implies the inequality above.

Rewriting, we have
\[
 |u_k^i(x) + \a_k^i ( x_n - \e^i p_k x_n - \e^i \t \cdot x')|\leq C\a_k^i \e^i (\w^i + |x|^2).
\]
Up to an additional error term of order $(\e^i)^2$, this may be rewritten as
\[
 |u_k^i(x) + \a_k^i (1-\e^i p_k) ( x_n - \e^i \t \cdot x')|\leq C\a_k^i \e^i (\w^i + \e^i + |x|^2).
\]
Set $\nu = (-\e^i \t,1)/|(-\e^i \t,1)|$; then we have $|\nu -e_n|\leq C(n)\e^i$ and $|\nu - (-\e^i,\t)|\leq C(n)\e^i$. As this means that $|x_n-\e^i\t \cdot x' -x\cdot \nu|\leq C\e^i|x|\leq C((\e^i)^2 +|x|^2)$, our estimate gives
\[
 |u_k^i(x) + \a_k^i (1-\e^i p_k) x\cdot \nu |\leq C\a_k^i \e^i (\w^i + \e^i + |x|^2).
\]

Define $\a_k^{i,*} = \a_k^i (1-\e^i p_k)$ for $k = 2,\ldots,M$, $\a_K^{i,*}=\a_k^i$ for $k>M$, and $\a^{i,*,t}_1 = \a_1^i (1-\e^i p_1 + t)$. Then
\begin{align*}
 \xi_1^i &(\a_k^{1,*,t})^2 + \sum_{i=2}^N (\a_k^{i,*})^2 \xi_k^i\\
&= \sum_{k=1}^M \xi_k^i (\a_k^i)^2 [ 1 - 2\e^i p_k +(\e^i p_k)^2] +\sum_{k=M+1}^N \xi_k^i (\a_k^i)^2 + \xi_1^i (\a_1^i)^2 t (2-2\e^i p_1 +t)  \\
&= \xi_0^i + \sum_{k=1}^M \xi_k^i (\a_k^i)^2 [ - 2\e^i p_k +(\e^i p_k)^2] + \xi_1^i (\a_1^i)^2 t (2-2\e^i p_1 +t)\\
&= \xi_0^i + S +  \xi_1^i (\a_1^i)^2 t (2-2\e^i p_1 +t).
\end{align*}
Let us show that $S$ is small. Indeed, we know that
\[
 |\sum_{k=1}^M \xi_k^i (\a_k^i)^2 p_k| =  |\sum_{k=1}^M [\xi_k^i (\a_k^i)^2 - \xi_k \a_k^2]  p_k|\leq C\w^i. 
\]
It follows that
\[
 |S|\leq C\e^i \w^i.
\]
Recall that both $\xi_1^i$ and $\a_1^i$ are bounded above and below by universal constants; it follows that we may select a $t$ with $|t|\leq C\e^i \w^i$ such that 
\[
 \xi_1^i (\a_1^i)^2 t (2-2\e^i p_1 +t) = -S.
\]
We name $\a_1^{i,*}=\a_1^{i,*,t}$ for this $t$. 

Our previous estimate may now be written as
\[
 |u_k^i(x) + \a_k^{i,*} x\cdot \nu |\leq C_* \a_k^{i,*} \e^i (\w^i + \e^i + |x|^2),
\]
where $\nu$ is a unit vector with $|\nu-e_n|\leq C_2\e^i$, while the $\a_k^{i,*}$ satisfy (1) and have $|\a_k^{i,*} - \a_k^i|\leq C_2 \a_k^i \e^i$ (the constants $C_2,C_*$ are universal). Now select $\r_0<1/4$ so that $\r_0 C_* <1/8$, and then $i$ so large that $C_*(\w^i + \e^i) < \r/8$; we have shown that
\[
 |u_k^i(x) + \a_k^{i,*} x\cdot \nu |\leq \r/4 \a_k^{i,*} \e^i,
\] 
on $B_{\r}\cap \W^i$, which implies the theorem.
\end{proof}

\section{Iteration Procedure}\label{sec:iteration}

We begin with a lemma about solutions to a Poisson equation.

\begin{lemma}\label{lem:pullup} There exists $\r_D$ small such that for any $\r<\r_D$, there is an $\e_D(\r)$ such that for any $\e<\e_D$ the following holds: let $\W$ be an open subset of $B_1$ with $0\in \p \W$,  $\p \W\ss \{|x_n|\leq \e\}$, and $(0,-1/2)\in \W$. Let $v$ be a continuous function on $\W$, vanishing on $\p \W$, and satisfying $|v|\leq 1$ and $|\triangle v|\leq \e^2$. Assume that for some $x_0\in B_\r$, $v(x_0)\geq \frac{1}{4}\r$. Then there is an $\a$ with $c_D\leq \a \leq C_D$ such that on $B_{\sqrt{\e}}$, we have that 
\[
 |v(y) + \a y_n| \leq C_D\e.
\]
The constants may depend on $\r$ and universal quantities.
\end{lemma}

\begin{proof} 
We will always assume that $\e_D \ll \r^2$. First, by comparison with a function $q_1$ solving the Dirichlet problem
\[
 \begin{cases}
  \triangle q_1 = -\e^2 & \text{ on } \{x\in B_1 : x_n<  \e\}\\
  q_1 = 0 & \text{ on } \{x\in B_1 : x_n= \e\}\\
  q_1 = 1 & \text{ on } \{x\in \p B_1 : x_n <\e\},
 \end{cases}
\]
we have that $|v|\leq C(x_n -\e)^-$ on $B_{1/2}$. One consequence is that $(x_0)_n \geq c \r$, so that $B_{c \r}(x_0)\ss \W$. Note that for any $y\in B_{1/4}$ with $y_n<-2\e$, we have $B_{|y_n|/2}(y)\ss \W$ and that $|v|\leq C |y_n|$ on that ball, so from standard regularity this gives $|\n v(y)|\leq C$. In particular, $|\n v|\leq C$ on $B_{c\r}(x_0)$, and so $v\geq 1/8 \r$ on some smaller ball $B_{c_0\r}(x_0)$ (for convenience, we assume $c_0<1$).

Our first objective will be to show that for any $\d>0$, there is a $\r_\d$ such that for $\r<\r_\d$, the assumptions of the lemma imply that  on $B_{10\r}$ we have $v\geq -\d \r$. We argue by contradiction, assuming that for some $\d>0$ no such $\r_\d$ can be found. This means there exists a sequence of functions $v^i$ and sets $\W^i$ as in the lemma, with $\r^i\rightarrow 0$ (recall $(\e^i)^{1/2} \leq r^i$). Let $u^i$ be the dilated functions
\[
 u^i = \frac{v^i(\r^i \cdot)}{\r^i},
\]
defined on the ball $B_{1/\r^i}$. From the assumptions, $|\triangle u^i| \leq \r^i (\e^i)^2$ on $\W^i/\r^i$, and our first estimate guarantees that $|u^i|\leq C (x_n - \e/\r^i)^- \leq C (x_n - \r^i)^-$.  On the other hand, we know that at some point $x^i$ in $B_1$, $u^i \geq 1/4$, and at some point $y^i$ in $B_{10}$, $u^i<-\d$.

Along a subsequence, the functions $u^i$ converge locally uniformly on the half-space $H=\{x_n<0\}$ to a harmonic function $u$, which satisfies the estimate $|u(x)|\leq Cx_n^-$. Applying Liouville's theorem to the extension of $u$ by odd reflection, we see that $u = a x_n$ for some $|a|\leq C$. We also, however, have that along a further subsequence $x^i\rightarrow x^\8 \in B_1\cap H$ and $y^i \rightarrow y^\8\in B_{10}\cap H$, and that $u(x^\8)\geq 1/4$ while $u(y^\8)\leq -\d$. This clearly contradicts our expression for $u$.

Next, we solve several auxiliary Dirichlet problems. Let $\phi:[0,\8)\rightarrow [0,1]$ be a smooth cutoff function which vanishes on $[0,2]$ and is $1$ on $[3,\8)$.  The first is a harmonic function $q_2$ which satisfies
\[
 \begin{cases}
  \triangle q_2 = 0 & \text{ on } U:=(B_{10 \r}\cap \{x_n<\e\})\sm B_{c_0 \r}(x_0)\\
  q_2 = -2\d\r \phi(|x'|/\r) & \text{ on } B_{10\r} \cap \{x_n=\e\}\\
  q_2 = -2\d\r \text{ on } \p B_{10 \r}\cap \{x_n<\e\}
  q_2 = 0 & \text{ on } \p B_{c_0 \r}(x_0)
 \end{cases}
\]
The main fact we need about $q_2$ is that $|\n q_2|\leq C_2 \d$, uniformly in $\r,x_0$, and $\e$; this is easy to see after dilating by a factor of $\r$. Then we construct
\[
\begin{cases}
  \triangle f = \e^2 & \text{ on } U\\
  f = 0  & \text{ on } \p (B_{10 \r}\cap \{x_n<\e\}) \\
  f = v  & \text{ on } \p B_{c_0 \r}(x_0),
 \end{cases} 
\]
\[
\begin{cases}
  \triangle g_t = -\e^2 & \text{ on } U\\
  g_t = t  & \text{ on } \p (B_{10 \r}\cap \{x_n<\e\}) \\
  g_t = v  & \text{ on } \p B_{c_0 \r}(x_0),
 \end{cases} 
\]
and
\[
\begin{cases}
  \triangle v_* = 0 & \text{ on } U\\
  v_* = v  & \text{ on } \p (B_{10 \r}\cap \{x_n<\e\}) \\  
  v_* = 0  & \text{ on } \p B_{c_0 \r}(x_0).
 \end{cases} 
\]

For $v_*$, we have that $|v_*|\leq C(x_n-\e)^-$, and in particular that $|v^*|\leq C \e$ on $\{|x_n\leq \e\}$. Moreover, on $B_{c\r}(0,\e)$  we have the Taylor expansion
\[
 |v_*(x) -  a_* (x_n-\e)^-|\leq C|x-(0,\e)|^2,
\]
where $|a_*|\leq C_1$.

For the function $g_t$, we first choose a useful value for the parameter $t\in [-\r,0]$. To do so, note that we have $|\n g_t|\leq C$ uniformly in $\e,\r,x_0,$ and $t$ on the set $\{x\in B_{10\r}: x_n\in (-c\r,\e)\}$ (for some universal $c$ chosen so that this set does not intersect $B_{c_0 \r}$). In particular, by choosing $t= - C^* \e$ we can ensure that on $\{x\in B_{10\r}: x_n\in (-\e,\e)\}$, we have $g_t + v_*\leq 0$. Denote this function by $g$. Near $(0,\e)$, this function admits the Taylor series
\[
 |g(x) + C^*\e -  a_- (x_n-\e)^-|\leq C|x-(0,\e)|^2
\]
with $c_1\leq a_-\leq C_1$ , and all constants universal. The bounds on $a_-$ follow by noting that on $\p B_{c_0 \r}$, $c\r \leq v\leq C\r$, and applying an appropriate comparison argument.

Likewise for $f$, we have the Taylor series
\[
 |f(x) -  a_+ (x_n-\e)^-|\leq C|x-(0,\e)|^2
\]
with $c_1\leq a_+\leq C_1$. 

At this point we fix $\d$ so that $C_2\d$ (which bounds $|\n q_2|$) is less than $c_{1}/2$. Then we select the corresponding $\r_\d$ as $\r_D$, which ensures that $v\geq -\d \r$ on $B_{10\r}$. 

By the maximum principle, we have that $g+q_2 \leq g+ v_* \leq v \leq f + v_*$ on $U$. The first of these implies that $a_* + a_- \geq -\p_n q_2(0,\e) + a_- \geq c_1/2$. Consider the difference $g-f$: this has Laplacian bounded by $2\e^2$ on $U$, and is controlled by $C^*\e$ on $\p U$. It follows that $|\n (g-f)|\leq C\e/\r =C\e$, and so in particular $|a_+-a_-|\leq C\e$. We then have
\[
 -C^*\e + (a_* +a_-) (x_n-\e)^- -C|x-(0,\e)|^2 \leq  v \leq (a_* +a_+) (x_n-\e)^- +C|x-(0,\e)|^2
\]
on $B_{c\r}(0,\e)$, which implies that
\[
 |v - (a_*+a_-) x_n| \leq C\e
\]
on $B_{\sqrt{\e}}$.
\end{proof}

We now prove our main regularity theorem.

\begin{theorem}\label{thm:main} There is an $\e_R>0$ such that for any flat $(\W,\{u_k\})\in \cQ(N,Q)$ with flatness $f<\e_R$, $\osc_{B_1}\xi_0 \leq \e_R$, and $|\D u_k|\leq \e_R$, we have that for some $\nu\in S^{n-1}$, $\p\W\cap B_{1/2}$ is the graph of a $C^{1,\a}$ function $\g:B_{1/2}\cap \{x\cdot \nu=0\}\rightarrow \R$, with
\[
 \|\g\|_{C^{1,\a}} \leq C,
\]
and $\a,C$ universal.
\end{theorem}

\begin{proof} We will show that there is some unit vector $\nu$ such that for $s<1/4$, $B_s \cap \p \W$ is contained in the region $|x\cdot \nu|\leq C s^{1+\a}$. This implies the conclusion of the theorem by standard arguments: applying to translates of $\W$ and letting $\nu_x$ be the normal vector corresponding to $x\in \p\W\cap B_{1/2}$, we see that there is a unique $x\in \p \W$ with a given orthogonal projection $x'\in \{x\cdot \nu_0=0\}$, and that $|\nu_x - \nu_y|\leq C |x-y|^\s$. 

To simplify statements, below we adopt the convention that when an inequality about the functions $u_k$ or their derivatives is satisfied on a ball $B_r$, we mean that it holds for any $x\in B_r\cap \W$.

Let $\r = \min\{\r_0,\r_D,\frac{1}{16}\}$, and select the corresponding $\e_D,\e_F$ from Theorem \ref{thm:flatimp} and Lemma \ref{lem:pullup}. Take the constants $c_D,C_D$ as in Lemma \ref{lem:pullup}, and assume $C_D/c_D\geq C'\geq 2$, with $C'$ to be chosen. Set $\e_I = \min\{\e_D,\e_F\}$, and $\e_R = \frac{1}{2}(\e_I c_D^N/C_D^N)^{4\cdot 2^N}\ll \e_I$. Set $\w_0= \1\frac{\e_I c_D^N}{C_D^N}\2^{2^N} = (2 \e_R)^{1/4}$.

From the flatness assumption, we know there is a unit vector $\nu_0$ and numbers $\{\a_k^0\}$ satisfying
\[
 \sum_{i=1}^N (\a_k^0)^2\xi_k = \xi_0(0)
\]
and $|u_k + \a_k^0 x\cdot \nu_0| \leq \e_R$ on $\W$.  We split the indices into two categories: either we have the pair of inequalities (up to switching signs of those $u_k$)  $|u_k + \a_k^0 x\cdot \nu_0|\leq \a_k^0 \w_0$ and $|\triangle u_k|\leq \a_k^0 \w_0^2$ (we say then that $k\in \cA_0$) or not (so $k\in \cB_0$). Note that if $|\a_k^0|\geq \frac{1}{2} \w_0^2$, then  $k\in \cA_0$, so in particular $\cA_0$ is nonempty. For the indices in $\cB_0$, we instead have that $|u_k|\leq |\a_k^0| + \e_R <\w_0^2/2 + \w_0^4/2<\w_0^2$. Set $m_0=0$.

In particular, we have that for each $k\in \cA_0$,
\[
 |u_k + \a_k^0 x\cdot \nu_0|\leq \a_k^0 \w_0\leq \a_k^0 \e_I
\]
and
\[
 |\triangle u_k|\leq \a_k^0 \w_0^2 \leq\a_k^0 \e_I^2,
\]
while for each $k\in \cB_0$,
\[
 |u_k|< \w_0^2 \leq \e_I^2
\]
and
\[
 |\triangle u_k|\leq \w_0^4 \leq \e_I^4.
\]
We will now iteratively construct a sequence of radii $r_i$, numbers $\a_k^i$, collections of indices $\cA_i,\cB_i$, integers $m_i$, and unit vectors $\nu_i$. Let us assume as an inductive hypothesis the following:
\begin{enumerate}
 \item We have $\cB_i\ss \cB_{i-1}$; if $\cB_i=\cB_{i-1}$ then $m_i=m_{i-1}$, and otherwise $m_i=m_{i-1}+1$.
 \item The numbers $\a_k^i$ satisfy
\[
 \sum_{k=1}^N\xi_k (\a_k^i)^2 = \xi_0 (0).
\]
 \item We have the following bounds on $B_{r_i}$ for $k\in \cA_i$:
\[
 |u_k + \a_k^i x\cdot \nu_i|\leq \a_k^i r_i \w_i
\]
and
\[
 |\triangle u_k| \leq \a_k^i \w_i^2/r_i
\]
 \item We have the following bounds on $B_{r_i}$ for $k\in \cB_i$:
\[
 |u_k| < \w_i^2 r_i
\]
and
\[
 |\triangle u_k|\leq \w_i^4/r_i.
\]
 \item If $m_i = m_{i-1}$, then $\w_i\leq \w_{i-1}/2$ and $r_i = \r r_{i-1}$. If not, then $\w_i = C_D/c_D \sqrt{\w_{i-1}}$ and $r_i = r_{i-1}\sqrt{\w_{i-1}}$. In either case, $\w_i \leq \e_I$.
 \item We have $|\nu_i - \nu_{i-1}|\leq C \w_i$.
 \item In addition, $\w_i \leq ( e_I (c_D/C_D)^{|\cB_i|})^{2^{|\cB_i|}}$.
\end{enumerate}
We will first show, by induction, that such a sequence may be constructed, and then explain how this implies the original claim.

Assume that the first $i$ elements of the sequence have been constructed. Then for every $k\in \cB_i$, check whether the inequality
\[
 |u_k|< \w_i^2/4 \r r_i
\]
holds on $B_{\r r_i}$. First consider the case when it does. Then we set $m_{i+1}=m_i$, $\cB_{i+1}=\cB_{i}$, $r_{i+1}=\r r_i$, and $\w_{i+1} = \w_i/2$ (this guarantees (1), (5), and (7)). Then apply Theorem \ref{thm:flatimp} to $(\W/r_i, \{u_K(r_i \cdot)/r_i\})$, using the fact that $\w_i \leq \e_I$. This gives a vector $\nu_{i+1}$ and numbers $\{a_k^{i+1}\}$ satisfying (6) and (2) respectively, and in addition $|\a_k^{i+1}-\a_k|\leq C|\a_k| \w_i$, for which 
\[
 |u_k + \a_k^{i+1} x\cdot \nu_{i+1}|\leq \a_k^{i+1} r_{i+1} \w_{i+1}
\]
on $B_{r_{i+1}}$. To check the other inequality in (3), note that
\[
 |\triangle u_k|\leq \a_k^i \w_i^2 /r_i \leq \a_k^i 4 \w_{i+1}^2/r_{i+1} \leq \a_k^{i+1} \w_{i+1}^2/r_{i+1},
\]
using simply that $\r\leq \frac{1}{16}$. We already have the first inequality of (4) by assumption, while for the second one,
\[
 |\triangle u_k|\leq  \w_i^4 /r_i \leq  \w_{i+1}^4/r_{i+1}.
\]
 
Now consider the case when for at least one $k\in \cB_i$, we have 
\[
 |u_k|\geq \w_i^2/4 \r r_i
\]
on $B_{\r r_i}$. Let $\cB_{i+1}$ contain all the $k\in \cB_i$ for which we still have
\[
 |u_k|< \w_i^2/4 \r r_i
\]
on $B_{\r r_i}$. Set $m_{i+1} = m_i +1$, $\r_{i+1} = \r_i \sqrt{\w_i}$, $\w_{i+1}=C_D/c_D \sqrt{\w_i}$, and $\nu_{i+1}=\nu_i$. This means (6) and (1) hold automatically. We also have (7):
\[
 \w_{i+1} \leq \frac{C_D}{c_D} ( e_I (c_D/C_D)^{|\cB_i|})^{2^{|\cB_i|-1}}\leq ( e_I (c_D/C_D)^{|\cB_i|-1})^{2^{|\cB_i|-1}}\leq ( e_I (c_D/C_D)^{|\cB_{i+1}|})^{2^{|\cB_{i+1}|}},
\]
and as this implies $\w_{i+1}\leq \e_I$, we also have (5). Note that from the first part of (3) for $i$, we have that $\p \W \cap B_{r_i}\ss \{|x\cdot \nu_i|\leq \w_i r_i\}$. Then we apply Lemma \ref{lem:pullup} to $(\W/r_i, \frac{u_k(r_i\cdot)}{\w_i^2 r_i})$ for each $k\in \cB_i \sm \cB_{i+1}$; this tells us that there is an $\a_k^{i+1}$ with $c_D\w_i^2 \leq \a_k^{i+1} \leq C_D \w_i^2$ such that on $B_{r_i+1}$, we have
\[
 |u_k + \a_k^{i+1} x \cdot \nu_i| \leq \frac{C_D}{c_D} \a_k^{i+1} r_i \w_i = \a_k^{i+1} r_{i+1} \w_{i+1}.
\]
In addition, we have
\[
 |\triangle u_k|\leq \w_i^4/r_i \leq \a_k^{i+1}/c_D \w_i^2 /r_i = \a_k^{i+1}/c_D (\frac{c_D}{C_D})^5 \w_{i+1}^{5} /r_{i+1}\leq \a_k^{i+1} \w_{i+1}^2/r_{i+1}.
\]
Hence (3) holds for these indices. To define $\a_k^{i+1}$ for the other indices $k$, we recall that for some $q\in \cA_i$, $\a_q^i,\xi_q\geq c(n)$. We set $\a_k^{i+1}=a_k^i$ for $k\neq q, \notin \cB_i\sm \cB_{i+1}$, and then select $\a_q^{i+1}$ so that (2) holds. As $|\a_k^{i+1} -\a_k^i|\leq C \w_i^2$ for all $k\neq q$, this is also true for $k=q$. Let us then check (3) for $k\in \cA_i$:
\begin{align*}
 |u_k + \a_k^{i+1} x \cdot \nu_i|&\leq \a_k^i r_i \w_i + C|\a_k^i - \a_k^{i+1}| \a_k^{i+1} r_{i+1} \\
&\leq 2\a_k^{i+1} \frac{c_D}{C_D} r_{i+1}\w_{i+1} + C (\frac{c_D}{C_D})^4 \w_{i+1}^4 \a_k^{i+1} r_{i+1}\\
&\leq \a_k^{i+1} r_{i+1}\w_{i+1}
\end{align*}
provided $C'$ is large enough. Note that the second term is $0$ unless $k=q$, in which case $\a_k^{i+1}\geq c(n)$. For the other inequality,
\[
 |\triangle u_k|\leq \a_k^i \w_i^2/r_i \leq 2 \a_k^{i+1} (\frac{c_D}{C_D})^5 \w_{i+1}^5 /r_{i+1}\leq \a_k^{i+1}\w_{i+1}^2/r_{i+1},
\]
which establishes (3). Finally, we check (4) for every $k\in \cB_{i+1}$:
\[
 |u_k|< \w_i^2 r_i = (\frac{c_d}{C_D})^3 \w_{i+1}^{3} r_{i+1}<\w_{i+1}^2 r_{i+1},
\]
while
\[
 |\triangle u_k|\leq \w_i^4/r_i \leq \w_{i+1}^4/r_{i+1}.
\]
This concludes the induction argument.

Finally, we show that this implies our original claim. Let $\w(r)$ be the smallest number for which $\p \W\cap B_r \ss \{|x \cdot \nu|\leq r \w(r)\}$ for $\nu=\nu_i$ where $r\in (r_i,r_{i+1})$. Then there is a trivial estimate $\w(s) \leq \frac{r}{s} \w(r)$ for $s\leq r$, and also we know that $\w(r_i)\leq \w_i$. Let $m = \lim_{i\rightarrow \8} m_i\leq N$; we argue by induction on $m$  that there are constants $C,\a$ (which may depend on $m$) such that $\w(r)\leq C(\w_0 r)^\a$. For $m=0$, this is standard. For $m\geq 1$, let $l$ be the first integer for which $m_l>m_{l-1}$. By induction, we then have
\[
 \w(r) \leq C (\w_l r/r_l)^\a.
\]
for $r\geq r_l$. Now, for $i<l$ we have $\w(\r^i)\leq \w_0 2^{-i}$, from which it follows that $\w(r)\leq \w_0 \r^{-1} r^{-\log 2/\log \r} = \w_0/\r r^{\a_0}$ for $r\leq r_{l-1}$. For $r\in [r_{l-1},r_l]$, we have (set $\b = \frac{1-\a_0/2}{1+\a_0/4}$)
\begin{align*}
 \w(r) &\leq \w(r_{l-1}) \frac{r_{l-1}}{r_{l}} \\
&= \w(r_{l-1}) (\frac{r_{l-1}}{r_{l}})^\b (\frac{r_{l-1}}{r_{l}})^{1-\b}\\
&= \w_{l-1}^{1-\b/2} (\frac{r_{l-1}}{r_{l}})^{1-\b}\\
&  \leq (\w_0/\r)^{1-\b/2} r_{l-1}^{\a_0(1-\b/2)}(\frac{r_{l-1}}{r_{l}})^{1-\b}\\
& \leq (\w_0/\r)^{1-\b/2} r^{\a_0 (1-\b/2)}.
\end{align*}
We thus have, for $r \leq r_l$, that $\w(r) \leq C (\w_0 r)^{\a_0(1-\b/2)}$, and that for $r\geq r_l$,
\[
 \w(r) \leq C (\w_l r/r_l)^\a \leq C (\w_0 r_l)^{\a \a_0 (1-\b/2)} (\frac{r}{r_l})^{\a \a_0 (1-\b/2)} = C(\w_0 r)^{\a \a_0 (1-\b/2)}. 
\]
It follows that
\[
 \w(r)\leq C(\w_0 r)^{\a \a_0 (1-\b/2)}
\]
for all $r$, completing the argument.

It now also follows that $|\nu_i - \nu_{i+1}|\leq C(\w_0 r_i)^\a$, and so $\nu_i\rightarrow \nu$, with $|\nu_i - \nu| \leq C(\w_0 r_i)^\a$. We conclude that
\[
\max\3 t: \p \W\cap B_r \ss \{|x \cdot \nu|\leq r t\}\4 \leq C(\w_0 r_i)^\a,
\]
which gives our original claim.
\end{proof}

\section{The Singular Set}\label{sec:sing}

In this section we discuss the size of $\p \W \sm \p^* \W$ for a minimizer (in the sense of Section \ref{sec:blowup}). As of this point, we know that $\p^* \W$ is a union of $C^{1,\a}$ graphs, and
\[
-\triangle u_k = (u_k)_\nu d\cH^{n-1}\mres_{ \p^*\W} + \l_k u_k,
\]
with the $u_k$ satisfying property E at each point in $\p^* \W$ :
\[
 \sum_{k=1}^N \xi_k (u_k)_\nu^2 = \xi_0.
\]
To avoid technicalities, we will assume $\xi_0$ constant here. The following proposition is based on a formula of Georg Weiss \cite{Weiss}.

\begin{proposition}Take $0\in \p \W$, and perform a blow-up $u_k(r_i \cdot)/r_i \rightarrow v_k$, $\W/r_i \rightarrow V$ as in Lemma \ref{lem:blowup}. Then $V$ is a cone, the $v_k$ are $1-$ homogeneous, and we have $v_k = \a_k v_1$ for all $k$.
\end{proposition}

\begin{proof} Let $\phi(r)$ be Weiss's energy
\[
\phi(r) = \frac{1}{r^n} \int_{B_r\cap \W} \1\sum_{k=1}^N \xi_k |\n u_k|^2 + \xi_0\2  - \frac{1}{r^{n+1}}\int_{\p B_r} \sum_{k=1}^N \xi_k u_k^2.
\]
As the $u_k$ are Lipschitz, this is an absolutely continuous function. We will check that
\[
\phi'(r) \geq - C r + \frac{2}{r^{n+2}}\int_{\p B_r} \sum_{k=1}^N \xi_k (u_k - r(u_k)_r)^2. 
\]
Let us first see, however, how this would imply the conclusion. We have that $\lim_{r\rightarrow 0} \phi(r)=\phi_0$ exists (as $\phi + C r^2/2$ is monotone). It follows that
\[
\lim_{i}\phi(r r_i) - \phi(s r_i) = 0
\]
for any $s<r$, and so
\[
\lim_i \int_{s r_i}^{r r_i}  \frac{2}{t^{n+2}}\int_{\p B_t} \sum_{k=1}^N \xi_k (u_k - t(u_k)_r)^2 dt = 0.
\]
On the other hand, if we rescale (and set $v_k^i(x)=u_k(r_i x)/r_i$), this gives
\[
\lim \int_{s}^r \frac{2}{t^{n+2}}\int_{\p B_t}\sum_{k=1}^N \xi_k (v_k^i - t(v_k^i)_r)^2 dt = 0.
\]
Using the dominated convergence theorem, this gives
\[
\int_{s}^r \frac{2}{t^{n+2}} \int_{\p B_t}\sum_{k=1}^N \xi_k (v_k - t (v_k)_r)^2 dt = 0.
\]
This is true for any $s<r$, so $t(v_k)_r =  v_k$; this implies that the $v_k$ are homogeneous of degree one and that $V$ is a cone. In fact, $V$ must be a connected cone: were it disconnected, an application of the Alt-Caffarelli-Friedman monotonicity formula to $v_1$ would give that $v_1 = \a_1 |x\cdot \nu|$ for some unit vector $\nu$, and in particular $|B_r\cap V|=|V|$. This contradicts the density estimates from Lemma \ref{lem:stronglb}. To see that each $v_k$ is a multiple of $v_1$, note that $v_1$ restricted to the sphere $\p B_1$ is a positive eigenfunction of the Laplace-Beltrami operator on the connected region $V\cap \p B_1$. Thus $v_1$ is the first eigenfunction, and its corresponding eigenvalue is simple. Each of the $v_k$ solves the same eigenvalue problem as $v_1$, however, with the same eigenvalue (the eigenvalue depends only on the degree of homogeneity, see \cite[Chapter 12]{CS}). Thus we must have $v_k = \a_k v_1$ (it is possible that $\a_k=0$).

It remains to check the monotonicity formula. For this, let $T$ be the vector field
\[
T = x (\xi_0 + \sum_k \xi_k|\n u_k|^2)  - 2 \sum_k \xi_k \n u_k \cdot x \n u_k.
\]
This is a bounded vector field on $\W$, which is smooth on the interior. In particular, we have
\[
\dvg T = n\xi_0 + \sum_k \xi_k \1(n-2) |\n u_k|^2 + 2\l_k u_k \n u_k \cdot x\2. 
\]
Near each point of $\p^* \W$, we have that
\[
T\cdot \nu = 2 x\cdot \nu \xi_0 - 2 x \cdot \nu \sum_k \xi_k |\n u_k|^2 = 0.
\]
Applying the divergence theorem (see \cite{CTZ}) to $T$ on $B_r \cap \W$, we recover a Rellich identity:
\[
\int_{B_r} n\xi_0 + \sum_k \xi_k [(n-2) |\n u_k|^2 + 2\l_k u_k \n u_k \cdot x] = r \int_{\p B_r} \xi_0 + \sum_k \xi_k [|\n u_k|^2 - 2 ((u_k)_r)^2]. 
\]

Now we estimate the derivative of $\phi$. We have
\begin{align*}
\phi'(r) &= \frac{1}{r^n}\int_{\p B_r} (\xi_0 + \sum_k \xi_k |\n u_k|^2) - \frac{n}{r^{n+1}} \int_{B_r} (\xi_0 + \sum_k \xi_k |\n u_k|^2)\\
& \qquad - \sum_k \xi_k [\frac{1}{r^{n+2}} \int_{\p B_r}  2 r u_k (u_k)_r -  2 u_k^2]\\
& = \sum_k \xi_k[\frac{1}{r^n}\int_{\p B_r}  ((u_k)_r)^2 - \frac{2}{r^{n+1}} \int_{B_r} |\n u_k|^2 -\l_k u_k \n u_k \cdot x]\\
& \qquad - \sum_k \xi_k [\frac{1}{r^{n+2}} \int_{\p B_r}  2 r u_k (u_k)_r -  2 u_k^2] \\
& \geq \sum_k \xi_k[\frac{1}{r^{n+1}}\int_{B_r} -2\l_k u_k \n u_k \cdot x + \frac{2}{r^{n+2}} \int_{\p B_r}  (r (u_k)_r -   u_k)^2].
\end{align*}
We differentiated $\phi$, substituted in our identity, and finally integrated by parts and completed the square (dropping the favorable error term, hence the inequality). To conclude, simply note that $|u_k|\leq Cr$, and so the first term is at least $-Cr$.
\end{proof}

In the theorem below, an Alt-Caffarelli minimizer is a local minimizer on $B_1$ of the functional
\[
\int |\n u|^2 + |\{u>0\}|
\]
over nonnegative $H^1$ functions. We say that a Lipschitz function is stationary for the Alt-Caffarelli problem if $u$ is harmonic when positive, $\{u>0\}$ is a set of finite perimeter satisfying density estimates as in Lemma \ref{lem:stronglb}, $u$ has the property that $\max_{B_r(x)}u \geq c r$ for any $x\in \p \{u>0\}$, and $u_\nu = -1$ on $\p \W$ in the distributional sense. 

\begin{theorem}\label{thm:sing} Let $\W$ be a minimizer. Then for $n_*=2$, we have that:
\begin{enumerate}
	\item If $n\leq n_*$, then $\p \W = \p^* \W$.
	\item If $n= n_* +1$, then $\p \W\sm \p^*\W$ consists of at most finitely many points.
	\item For any $n$, $\cH^{n-n_*-1 +\t}(\p \W \sm \p^*\W)=0$ for any $\t>0$.
\end{enumerate}
If in addition $\W$ enjoys property S, then we may instead take $n_*$ the highest dimension for which Alt-Caffarelli minimizers are known to have no singular points (this is at least $4$, from \cite{JS}). 
\end{theorem}

\begin{proof}
	We have, at each blowup, that $v_k = \a_k v_1$ are homogeneous. Then Property E implies that $\sum_k \xi_k \a_k (v_1)_\nu^2 = \xi_0$. Together with the previously established properties, we have that a multiple of $v_1$ is stationary for the Alt-Caffarelli problem. It was shown in \cite{AC} that in dimension 2, the only such solution is of the form $(x\cdot \nu)^-$ for some $\nu$ (indeed, this is easy to check for homogeneous solutions). The result then follows from Federer's dimension reduction argument (see \cite{ Weiss} for details).
	
	If $\W$ satisfies property S, then we also have the minimality property of Lemma \ref{lem:blowup}, which implies that a multiple of $v_1$ is a local minimizer of the Alt-Caffarelli problem. These are of the form $(x\cdot \nu)^-$ up to dimension $n_*$, and we again refer to the dimension reduction argument as explained in \cite{Weiss}.
\end{proof}

\section{Appendix: Higher Regularity}

In this section, we sketch the argument to show that the reduced boundary in the situations of Theorems \ref{thm:shapes} and \ref{thm:fb} is not only locally given by a $C^{1,\a}$ graph, but also that this is the graph of an analytic function. The arguments given here are not, in our opinion, original, but rather an application of the technique of Kinderlehrer, Nirenberg, and Spruck \cite{KNS}. However, there has recently been some doubt expressed in the literature (see \cite{MTV,DFS5}, where alternative approaches are proposed) about whether or not this method applies to showing higher regularity of, specifically, $C^{1,a}$ boundaries in Bernoulli-type problems. Indeed, while the two-phase problem is one of the main examples given in \cite{KNS}, it is assumed there that the boundary is $C^{2,\a}$ to start with. In some other examples, such as that of three minimal surfaces touching along a curve, only a $C^{1,\a}$ boundary is assumed originally, but there the entire transformed system can be put in divergence, or conormal, form; this is not the case for Bernoulli problems.

Below we argue that the method of \cite{KNS} is sufficient to imply analyticity directly, without any extra steps before performing the hodograph transform. As an aide, we will use a Schauder estimate for nondivergence-form elliptic systems, presented now. We make no effort to pursue the greatest possible generality, but will briefly comment on some extensions at the end.

Below, $H = \{(y',y_n)\in \R^n: y_n \geq 0\}$ is a half-space and $Q_r^+ = \{(y',y_n)\in H: |y'|< r, y_n<r\} $ is a relatively open cylinder. We say a function is in $C^{k,\a}_{\text{loc}}(E)$ if it lies in $C^{k,\a}(K)$ for all $K$ compactly contained in $E$.

\begin{lemma}\label{lem:schauder} Fix $\a\in(0,1)$. Let $v: H \rightarrow \R^N$ be a function in $C^{1,\a}(H)\cap C^{\8}_{\text{loc}}(H)$, with
\[
 \sup_{x\in H} (1+|x|)^{n-1} |v| + \sup_{x\in H} (1+|x|)^n |\n v| + \sup_{x,y\in H} (1+\min\{|x|,|y|\})^{n+\a} \frac{|\n v(x)-\n v(y)|}{|x-y|^\a}  < \8.
\]
 Assume that $v$ satisfies the system
\[
 \sum_{i,j=1}^n \sum_{k=1}^N A^{kl}_{ij} \p_{y_i}\p_{y_j} v^k = \dvg f^l
\]
at each point of $H \sm \p H$; here the $A_{kl}^{ij}$ are the coefficients of a constant rank-four tensor, and $f^l$ are vector fields in $C^\8_{\text{loc}}(H)$. Moreover, assume that on $\p H$, $v$ satisfies the boundary conditions
\[
 \begin{cases}
  v^k = 0 & k=1,\ldots,N-1\\
  \sum_{i=1}^{n}\sum_{k=1}^N b_i^k \p_{y_i}v^k = g
 \end{cases}
\]
for each $x\in \p H$; here $b_j^k$ is a constant matrix, while $g$ is a continuous function of compact support. Assume that $A_{ij}^{k,l},b_j^k$ satisfy (A-D) below. Then if
\[
M:=\|g\|_{C^{0,\a}} + \sup_{x\in H} (1+|x|)^n |f| + \sup_{x,y\in H} (1+\min\{|x|,|y|\})^{n+\a} \frac{|f(x)-f(y)|}{|x-y|^\a}  < \8,
\]
we have that
\[
\|\n v\|_{C^{0,\a}(H)} \leq C M.
\]
\end{lemma}

Here are the assumptions on system that the lemma calls for:

\begin{enumerate}[(A)]
 \item The tensor $A_{ij}^{kl}$ is triangular: $A_{ij}^{kl}=0$ whenever $k<l$.
 \item The system is uniformly elliptic: there is a number $\L>0$ such that for any $\xi\in \R^n$, the tensor $A_{ij}^{kl}$ satisfies
\[
 \L |\xi|^{2N} \leq |\det \{ \sum_{i,j=1}^N A_{ij}^{kl} \xi_i\xi_j\}_{kl}| \leq \L^{-1} |\xi|^{2N}.
\]
Assuming (A), this is equivalent to
\[
 \L'|\xi|^2 \leq \sum_{i,j=1}^N A_{ij}^{kk} \xi_i\xi_j \leq \L'^{-1} |\xi|^2
\]
for every $k$.
  \item The system, together with the boundary condition, is coercive, meaning the Complementing Condition of \cite{ADN2} is satisfied. If (A), (B) hold, then this is equivalent to $b_n^N\neq 0$ (i.e. that this is an oblique derivative condition).
  \item $A_{ij}^{kl} = A_{ji}^{kl}$; this is without loss of generality.
\end{enumerate}

\begin{proof} We separate the argument into two steps: in the first step, we reduce the problem to a homogeneous one, with $f^k=0$, by solving auxiliary scalar problems. Then in the second step, we will use a representation formula from \cite{ADN2} to close the argument.

First, take a vector field $f\in C^{\8}_{\text{loc}}(H)$ and with
\[
 M_0(f) := \sup_{x\in H} (1+|x|)^n |f| + \sup_{x,y\in H} (1+\min\{|x|,|y|\})^{n+\a} \frac{|f(x)-f(y)|}{|x-y|^\a}  < \8,
\]
and consider a scalar problem
\[
 \begin{cases}
  \sum_{i,j=1}^n A_{ij} \p_{y_i}\p_{y_j} v = \dvg f & \text{ on } H\sm \p H \\
  v = 0 & \text{ on } \p H \\
  |v|\rightarrow 0 & \text{ as } |x|\rightarrow \8. 
 \end{cases}
\]
This problem admits a solution, which may be described as follows: after a suitable linear change of variables $x = L y$ preserving the hyperplane $\p H$, the problem is equivalent to solving
\[
 \begin{cases}
  \triangle \tilde{v} = \dvg \tilde{f} & \text{ on } H\sm \p H\\
  |\tilde{v}| \rightarrow 0 & |x|\rightarrow \8\\
  \tilde{v} = 0 & \text{ on } \p H.
 \end{cases}
\]
Here $\tilde{v}(x) = v(L x)$. We may extend $\tilde{v}$ by an odd reflection about $\p H$, and $\tilde{f}$ by an even reflection (note that this preserves the $C^{0,\a}$ norm) to a solution of
\[
 \begin{cases}
  \triangle \tilde{v} = \dvg \tilde{f} & \text{ on } \R^n\\
  |\tilde{v}| \rightarrow 0 & |x|\rightarrow \8.
 \end{cases}
\]
 A solution to this problem is given by convolution of $f$ with the derivative of the fundamental solution $\Psi(x)$:
\[
 \tilde{v}(x) = \int \n \Psi(x-y)\cdot  \tilde{f}(y) dy.
\]
This is well-defined (using the decay assumption of $f$), and because $|\n \Psi(x)| \leq C |x|^{1-n}$ we also have $|\tilde{v}| \leq CM_0(f) (1 + |x|)^{1-n}$ . From \cite[Theorem 4.15]{GT}, we have the estimate (with $R = (|x|+1)/2$)
\begin{align*}
 R \sup_{B_R(x)} &|\n \tilde{v}| + R^{1+\a} [\n \tilde{v}]_{C^{0,\a}(B_R(x))}   \\
&\leq C[ \sup_{B_R(x)}|\tilde{v}| + R \sup_{B_{R}(x)}|\tilde{f}| + R^{1+\a} [\tilde{f}]_{C^{0,\a}(B_R(x))}]\\
& \leq C M_0(f) R^{1-n}.
\end{align*}
In particular, this gives
\[
 M_0(\n \tilde{v}) \leq C M_0 (f).
\]
Changing back to the original variables, we have found a solution to our scalar problem, and shown that it has the estimate
\[
 M_0(\n v) + \sup_H (1+|x|)^{1-n}|v| \leq C M_0(f).
\]
Note also that this $v\in C^\8_{\text{loc}}(H)$.

Next, let us use this to find a solution $u$ to the elliptic system (with $f^l$ now as in the statement of the lemma)
\[
  \begin{cases}
  \sum_{i,j=1}^n \sum_{k=1}^N A^{kl}_{ij} \p_{y_i}\p_{y_j} u^k = \dvg f^l & \text{ on } H\sm \p H\\
  u^k = 0 & \text{ on } H \sm \p H\\
  |u^k|\rightarrow 0 & \text{ as } |x|\rightarrow 0.
  \end{cases}
\]
Indeed, from (A) we have that this system is triangular, so the equation for $u^N$ is decoupled and we may solve it first, using the argument above. From the assumption on $f$,
\[
 M_0(f^l) \leq M,
\]
and from our estimate
\[
 M_0(\n u^N) \leq C M_0(f^N) \leq C M.
\]
Now the $l=N-1$ equation takes the form
\[
 \sum_{i,j}A^{N-1,N-1}_{ij} \p_{y_i}\p_{y_j} u^{l-1} = \sum_{i} \p_{y_i} [ f^{N-1}_i - \sum_j (A_{ij}^{N,N-1} \p_{y_j} u^N)] := \dvg f^{N-1}_*.
\]
As
\[
 M_0(f^{N-1}_*) \leq C[M_0(f^{N-1}) + M (\n u^N)] \leq C M,
\]
we may again solve a scalar equation to obtain $u^{N-1}$, which will have $M_0(u^{N-1}) \leq C M$. We continue inductively in this manner until we have found all of the $u^k$.

Set $w = v - u$. We know that $w\in C^\8_{\text{ loc }}(H)\cap C^{1,\a}(H)$, satisfies the bound $|w| \leq C (1 + |x|)^{1-n}$, and solves the system
\[
 \begin{cases}
 \sum_{i,j=1}^n \sum_{k=1}^N A^{kl}_{ij} \p_{y_i}\p_{y_j} w^k = 0 & \text{ on } H\sm \p H\\
  w^k = 0 & \text{ on } \p H \text{ for } k=1,\ldots,N-1\\
  \sum_{i=1}^{n}\sum_{k=1}^N b_i^k \p_{y_i}w^k = \tilde{g} & \text{ on } \p H. 
 \end{cases} 
\]
Here 
\[
\tilde{g} = g - \sum_{i=1}^{n}\sum_{k=1}^N b_i^k \p_{y_i}u^k_i,
\]
and we have $ M_0(\tilde{g}) \leq C M$ (abusing notation and using the same definition of $M_0(\cdot)$ for functions defined on $\p H$). We also have, from the assumption on $v$, that 
\begin{equation}\label{eq:apriori}
M_0(\n w) + \sup_H \frac{|w|}{(1+|x|)^{n-1}} <\8.
\end{equation}

We now enter into the second part of the proof, which is based on obtaining a representation formula for $w$. First of all, from \cite[Theorem 4.1, Corollary 4.1]{ADN2} for any smooth and compactly supported functions $\phi^l$, the elliptic system
\begin{equation}\label{eq:auxsystem}
 \begin{cases}
 \sum_{i,j=1}^n \sum_{k=1}^N A^{kl}_{ij} \p_{y_i}\p_{y_j} \w^k = 0 & \text{ on } H\sm \p H\\
  \w^k = \phi^k & \text{ on } \p H \text{ for } k=1,\ldots,N-1\\
  \sum_{i=1}^{n}\sum_{k=1}^N b_i^k \p_{y_i}\w^k = \phi^N & \text{ on } \p H 
 \end{cases} 
\end{equation}
admits a solution given by
\begin{equation}\label{eq:rep}
 \w^k (x',x_n) = \sum_{l=1}^N \int_{\p H}  K^{kl} (x' - y',x_n) \phi^l(y')dy',
\end{equation}
Where the ``Poisson kernels'' $K^{kl}$ enjoy the estimates
\begin{equation}\label{eq:ADNest}
\begin{cases} 
|D^s K^{kl}(x)| \leq C |x|^{1-n-s} & l=1,\ldots, N-1\\
|D^s K^{kN}(x)| \leq C |x|^{2-n-s}(1 + |\log|x||) &
\end{cases}
\end{equation}
with the logarithm omitted and kernel homogeneous unless $n=2$ and $s=0$. Our first step is to demonstrate that for any such $\phi$, the corresponding $\w$ admits the estimate
\begin{equation}\label{eq:repest}
 \sup_{x\in H} \frac{|\w(x)|}{|x|^{2-n}(1 + |\log(1+|x|)|)} + \|\n \w \|_{C^{0,\a}(H)}  \leq C [ M_0(\phi^N) + \sum_{l=1}^{N-1} M_0(\n \phi^l) + \sup_H \frac{|\phi^l|}{(1+|x|)^{n-1}}] .
\end{equation}
The logarithm may be omitted unless $n=2$. The first term is clear from combining \eqref{eq:rep} and \eqref{eq:ADNest} with $s=0$. The second estimate follows from \cite[Theorem 3.4]{ADN1} after writing 
\[
 \n \w^k(x',x_n) = \sum_{l=1}^N \int_{\p H}  \n K^{kl} (x' - y',x_n) \phi^k(y')dy',
\]
this is explained in \cite[Section 5]{ADN2}.

Now a simple approximation argument implies that the representation formula \eqref{eq:rep} and the estimate \eqref{eq:repest} remain valid for any $\phi^l$ for which the right-hand side of \eqref{eq:repest} is finite, even if they lack compact support.

Now fix $t>0$ and take $\phi^k(x') = w^k(x',t)$ for $k<N$, and $\phi^N(x') = \sum_{i=1}^{n}\sum_{k=1}^N b_i^k \p_{y_i}w^k(x',t)$. Using \eqref{eq:apriori}, we see that the right-hand side of \eqref{eq:repest} is finite for these $\phi$, and hence the representation formula produces functions $\w^k$ satisfying the equation \eqref{eq:auxsystem} which admit estimate \eqref{eq:repest}. Consider the difference $\z(x',s) = w(x',t+s) - \w(x',s)$: this is a solution to
\[
 \begin{cases}
 \sum_{i,j=1}^n \sum_{k=1}^N A^{kl}_{ij} \p_{y_i}\p_{y_j} \z^k = 0 & \text{ on } H\sm \p H\\
  \z^k = 0 & \text{ on } \p H \text{ for } k=1,\ldots,N-1\\
  \sum_{i=1}^{n}\sum_{k=1}^N b_i^k \p_{y_i}\z^k = 0 & \text{ on } \p H. 
 \end{cases} 
\]
We also know that (combining the estimates on $\w$ and $w$) that $|\z| \leq C(1+|\log|x||)$, and that $\z\in C^{\8}(H)$. Applying \cite[Theorem 8.3]{ADN2}, we see that $\z=0$, and hence
\[
\|\n w \|_{C^{0,\a}(\{(x',s):s\geq t\})}  \leq C [ M_0(\sum_{i=1}^{n}\sum_{k=1}^N b_i^k \p_{y_i}w^k(\cdot,t)) + \sum_{l=1}^{N-1} M_0(\n_{\R^{n-1}} w^l(\cdot,t)) + \sup_H \frac{|w(\cdot,t)^l|}{(1+|x|)^{n-1}}] . 
\]
Finally, send $t\rightarrow 0$: all of the terms on the right other than the first go to $0$, while the first term converges to $M_0(\tilde{g})\leq CM$. Hence we have shown that
\[
\|\n v\|_{C^{0,\a}(H)} \leq \|\n u\|_{C^{0,\a}(H)}+  \|\n w \|_{C^{0,\a}(H)}  \leq C M,
\]
as desired.
\end{proof}

The following is an elementary corollary, which localizes the estimate to domains and applies it to problems with variable coefficients.

\begin{corollary}\label{cor:schauder2} Fix $\a\in(0,1)$. Then there is a $\d>0$ such that the following holds: Let $v: Q_2^+ \rightarrow \R^N$ be a function in $C^{1,\a}(Q_2^+)\cap C^{\8}_{\text{loc}}(Q_2^+)$.
 Assume that $v$ satisfies the system
\[
 \sum_{i,j=1}^n \sum_{k=1}^N \p_{y_i} A^{kl}_{ij}(x) \p_{y_j} v^k = \dvg f^l
\]
at each point of $Q_2^+ \sm \p H$, and that on $\p H$, $v$ satisfies the boundary conditions
\[
 \begin{cases}
  v^k = 0 & k=1,\ldots,N-1\\
  \sum_{i=1}^{n}\sum_{k=1}^N b_i^k(x) \p_{y_i}v^k = g
 \end{cases}
\]
for each $x\in Q_2^+\cap \p H$; here $g,f$ are as in Lemma \ref{lem:schauder}, $b_j^k(x'),A_{ij}^{kl}(x)$ are as in that lemma for each $x\in Q_2^+$ or $x'\in Q_2^+\cap \p H$ and in addition $A\in C^\8_{\text{loc}}(Q_2^+)$, as well as
\[
 [A_{ij}^{kl}]_{C^{0,\a}(Q_2^+)} + [b_j^k]_{C^{0,\a}(Q_2^+\cap \p H)} \leq \d.
\]
Then
\[
\|\n v\|_{C^{0,\a}(Q_1^+)} \leq C [\|v\|_{C^{0,\a}(Q_2^+)} + \|f\|_{C^{0,\a}(Q_2^+)} + \|g\|_{C^{0,\a}(Q_2^+\cap \p H)}] .
\]
\end{corollary}

\begin{proof}
 Fix $\eta: H \rightarrow [0,1]$ a smooth cutoff which is $1$ on $Q_1^+$ and which is compactly supported on $Q_2^+$. Then a direct computation shows that $w = \eta v$ solves the following system on all of $H\sm \p H$:
\begin{align*}
 \sum_{i,j,k}& A^{kl}_{ij}(0) \p_{y_i}\p_{y_j} w^k = \dvg(\eta f^l) - \n \eta f^l \\
& + \sum_{i,j,k}\p_{y_i}[2 A_{ij}^{kl}(x) (\p_{y_j}\eta)v^k + (A_{ij}^{kl}(0)-A_{ij}^{kl}(x))\p_{y_j}w^k]  - v^k (\p_{y_i} A_{ij}^{kl}(x)\p_{y_j}\eta) \\
&= \tilde{h}^l + \dvg \tilde{f}^l.
\end{align*}
If we set
\[
 M = \|v\|_{C^{0,\a}(Q_2^+)} + \|f\|_{C^{0,\a}(Q_2^+)} + \|h\|_{C^{0,\a}(Q_2^+)}+\|g\|_{C^{0,\a}(Q_2^+\cap \p H)} + \d \|\n w\|_{C^{0,\a}}, 
\]
then both $\tilde{h}$ and $\tilde{f}$ are in $C^{0,\a}(H)$ and supported on $Q_2^+$, with norm controlled by $M$. The inhomogeneous boundary condition becomes
\[
 \sum_{i,k} b_i^k(0) \p_{y_i}w^k = \eta g + \sum_{i,k}[ v^k b_i^k(x) \p_{y_i}\eta + (b_i^k(0)-b_i^k(x))\p_iw^k ]  = \tilde{g},
\]
also with $\|\tilde{g}\|_{C^{0,\a}}\leq CM$.

Now extend $\tilde{h}^l$ to a compactly supported function on $\R^n$ with mean $0$ and $\|\tilde{h}\|_{C^{0,\a}}\leq CM$, and write
\[
 \z^l(x) = \int \Phi(x-y) \tilde{h}^l(y)dy,
\]
where $\Phi$ is the fundamental solution to the Laplace equation. This function satisfies
\[
 |\z^l(x)| \leq C M (1+|x|)^{1-n}, \qquad |\n \z^l| \leq C M (1+|x|)^{-n} \qquad |D^2 \z^l| \leq C M (1+|x|)^{-n-1};
\]
note the improved decay because of the mean-zero property. In particular, we may rewrite $\dvg \tilde{f}^l + \tilde{h}^l = \dvg \bar{f}^l$, where $\bar{f}^l = \tilde{f}^l + \n \z^l$ and has the decay property 
\[
 \sup_{x\in H} (1+|x|)^n |\bar{f}| + \sup_{x,y\in H} (1+\min\{|x|,|y|\})^{n+\a} \frac{|\bar{f}(x)-\bar{f}(y)|}{|x-y|^\a}\leq CM.
\]
Applying Lemma \ref{lem:schauder}, we obtain
\[
 \|\n w\|_{C^{0,\a}(H)} \leq C_0 M.
\]
Provided we now choose $\d < C_0/2$ and reabsorb the last term in $M$, this gives
\[
 \|\n w\|_{C^{0,\a}(H)} \leq 2C_0 [\|v\|_{C^{0,\a}(Q_2^+)} + \|f\|_{C^{0,\a}(Q_2^+)} + \|h\|_{C^{0,\a}(Q_2^+)}+\|g\|_{C^{0,\a}(Q_2^+\cap \p H)}],
\]
and implies the conclusion.
\end{proof}

Equipped with this estimate, we may now prove a Schauder-type theorem for an elliptic system which will be relevant to us shortly.

\begin{lemma}\label{lem:schauder3} Let $v: Q_3^+ \rightarrow \R^N$ be in $C^\8_{\text{loc}}(Q_3^+)\cap C^{1,\a}(Q_3^+)$, and solve the following quasilinear elliptic system:
\[
 \begin{cases}
  \sum_{i=1}^n \p_{y_i} [A_{i}^l (\n v)] = f^l & \text{ on } Q_3^+\sm \p H \\
  v^k = 0 & k=1,\ldots,N-1,  \text{ on } Q_3^+\cap \p H\\
  B(\n v) = 0 & \text{ on } Q_3^+ \cap \p H.
 \end{cases}
\]
Assume that $A,B$ are $C^{2}$ with $ \p_{p_j^k} A_i^l = 0$ if $k<l$, $ \L |\xi|^2 \leq \sum_{i,j} \p_{p_j^k} A_i^l \xi_j \xi_k \leq \L^{-1} |\xi|^2$, and $|\p_{p_n^N} B| \geq \L>0$. Assume also that $f^l\in C^{0,\a}(Q_3^+)$. Then $v \in C^{2,\a}(Q_1^+)$, with estimate
\[
 \|v\|_{C^{2,\a}(Q_1^+)}\leq C (\|v\|_{C^{1,\a}(Q_3^+)}, \|f\|_{C^{0,\a}(Q_3^+)})
\]
\end{lemma}

\begin{proof} First of all, we may write $f^l = \dvg z^l$ for some vector field $z^l$ with $\|z^l\|_{C^{1,\a}(Q_3^+)} \leq C \|f^l\|_{C^{0,\a}(Q_3^+)}$; this may be done by, for example, solving an appropriate boundary value problem for $-\triangle q^l = f^l$, and setting $z^l = \n q^l$.

Fix $e$ a unit vector in $\p H$, and consider the incremental quotients
\[
 v_h(x) = \frac{v(x+he)-v(x)}{h},
\]
for $|h|<\frac{1}{2}$. Unless otherwise indicated, all constants below will be independent of $h$. These $v_h$ are in $C^{1,\a}(Q_2^+)$, have
\[
 \|v_h\|_{C^{0,\a}(Q_2^+)} \leq C \|v\|_{C^{1,\a}(Q_3^+},
\]
and satisfy the following system:
\[
 \begin{cases}
  \sum_{i,j,k} \p_{y_i} [ a_{ij}^{kl} \p_{y_j} v_h^k ] = \dvg z_h^l & \text{ on } Q_2^+\sm \p H \\
  v^k_h = 0 & k=1,\ldots,N-1,  \text{ on } Q_2^+\cap \p H\\
  \sum_{j,k} b_j^k \p_{y_j} v_h^k = 0 & \text{ on } Q_2^+ \cap \p H.
 \end{cases}
\]
Here $a_{ij}^{kl}$ is given by
\[
a_{ij}^{kl}(x) = \int_0^1 \p_{p_j^k} A^l_i (t\n v(x+eh) + (1-t)\n v(x)) dt,
\]
and
\[
 b_j^k(x') = \int_0^1 \p_{p_j^k} B(t \n v(x+he)+(1-t)\n v(x) )dt,
\]
\[
 z_h^l(x) = \frac{z(x+eh)-z(x)}{h}.
\]
We see that
\[
 \|a_{ij}^kl\|_{C^{0,\a}(Q_2^+)} + \|b_j^k\|_{C^{0,\a}(Q_2^+)} + \|z_h\|_{C^{0,\a}(Q_2^+)} \leq C [\|v\|_{C^{1,\a}(Q_3^+)} + \|f\|_{C^{0,\a}(Q_3^+)}].
\]

Fix $r$ small, and consider the dilated function $v_{h,r,x}(y) = v_h(x+ry)$ for any $x\in Q_1^+\cap \p H$. This satisfies an equation  
\[
 \begin{cases}
  \sum_{i,j,k} \p_{y_i} [ a_{ij,r,x}^{kl} \p_{y_j} v_{h,r,x}^k ] = \dvg z_{h,r,x}^l & \text{ on } Q_2^+\sm \p H \\
  v^k_{h,r,x} = 0 & k=1,\ldots,N-1,  \text{ on } Q_2^+\cap \p H\\
  \sum_{j,k} b_{j,r,x}^k \p_{y_j} v_h^k = 0 & \text{ on } Q_2^+ \cap \p H.
 \end{cases}
\]
$a_{r,x},b_{r,x},$ and $z_{r,x}$ having the same assumptions, but in addition
\[
 [a_{ij,r,x}^kl]_{C^{0,\a}(Q_2^+)} + [b_{j,r,x}^k]_{C^{0,\a}(Q_2^+)}\leq Cr^\a [\|v\|_{C^{1,\a}(Q_3^+)} + \|f\|_{C^{0,\a}(Q_3^+)}].
\]
Choosing $r$ so small as to make the right-hand side less than the $\d$ in Corollary \ref{cor:schauder2}, we obtain the estimate
\[
 \|v_{h,r,x}\|_{C^{1,\a}(Q_1^+)} \leq C [\|v\|_{C^{1,\a}(Q_3^+)} + \|f\|_{C^{0,\a}(Q_3^+)}].
\]
On the other hand, for any $x\in Q_1^+$ with $|x_n|\geq r/2$, we may instead apply standard variational Schauder interior estimates to this problem (for example, inductively applying the scalar estimates in \cite{GT} first to $v^N$, then to $v^{N-1}$, and so on works) on a ball of radius $r/4$. Scaling both sets of estimates back, we learn that
\[
 \|v_h\|_{C^{1,\a}(Q_1^+)} \leq C(\|v\|_{C^{1,\a}(Q_3^+)},\|f\|_{C^{0,\a}(Q_3^+)}).
\]

This implies that the derivatives $\p_{x_j} \p_{x_i} v^k$, with $j\neq n$, have
\[
 \|\p_{x_j} \p_{x_i} v^k\|_{C^{0,\a}} \leq C(\|v\|_{C^{1,\a}(Q_3^+)},\|f\|_{C^{0,\a}(Q_3^+)}).
\]
On the other hand, we may write (with the help of the equation with $l=N$)
\[
 \sum_{i,j}\p_{p_j^N}A_i^N (\n v) \p_{y_i}\p_{y_j}v^N = f^N,
\]
with the coefficient $|\p_n^N A_n^N|\geq \L$ from the uniform ellipticity, and all coefficients H\"older-$\a$. It follows that 
\[
 \|\p_{x_n} \p_{x_n} v^N\|_{C^{0,\a}} \leq C(\|v\|_{C^{1,\a}(Q_3^+)},\|f\|_{C^{0,\a}(Q_3^+)}).
\]
Proceeding inductively for $k=N-1$, then $k=N-2$, until $k=1$, we see that all of the second derivatives of $v$ enjoy this estimate. This establishes the conclusion. 
\end{proof}

Finally, we are in a position to derive higher regularity to the free boundary problems of interest to us.

\begin{theorem}\label{thm:higherreg} Let $\W$ be as in Theorem \ref{thm:shapes}. Then the $C^{1,\a}$ graphs in the conclusion of the theorem are analytic. Let $\W$ be as in Theorem \ref{thm:fb}. Then $\p \W\cap B_{!/2}$ is analytic.
\end{theorem}

\begin{proof} In either situation, after rescaling and translating as needed it suffices to consider the following configuration: there is an open set $\W$, with $0\in \p \W$, $\p \W\cap B_{1}$ is a $C^{1,\a}$ graph over $H$, and $H$ is tangent to $\p \W$ at $0$. On $\W$ there are $N$ functions $\{u_k\}_{k=1}^N$, all of which are in $C^{1,\a}(\bar{\W}\cap B_1)$, vanish on $\p \W$, and are analytic on $\W$. At least one of these (say $u_N$) is strictly positive on $B_1$ and increasing in the $x_n$ direction (this follows from the free boundary condition, since $(u_k)_\nu(0)\neq 0$ for some $k$, and so $|(u_k)_n | \neq 0$ on a sufficiently small ball). The functions each satisfy an equation
\[
 -\triangle u_k = \l_k u_k
\] 
for some numbers $\l_k$ on $\W$, and the free boundary condition (after possibly multiplying each $u_k$ by a constant)
\[
\sum_{k=1}^N (u_k)_\nu^2 = 1
\]
on $\p \W$.

We perform the partial hodograph transform of \cite{KNS}. The mapping $\psi: (x',x_n)\mapsto (x',u_N(x',x_n))$ is a bijection of $\bar{\W}\cap B_1$ onto its image in $H$, and maps $\W$ to $\p H$. The mapping analytic on $\W$ and $C^{1,\a}$ on $\bar{\W}$. Let $v^N$ be the $n-th$ component of the inverse mapping, so as to make $\psi(y',v^N(y',y_n)) = (y',y_n)$. Define $v^k$, for $k=1,\ldots,N-1$, by $v^k(y) = u_k(\psi^{-1}(y))$. From the computations in \cite{KNS}, we see that the $v^k$ satisfy the following system on a neighborhood $Q_r^+$ of $0$ in $H$ (using $\p_i$ for $\p_{y_i}$):
\[
\begin{cases}
 \sum_{i=1}^{n-1}\p_i[(\p_n v^N) (\p_i v^k) -(\p_i v^N)(\p_n v^k) ] + \p_n [\frac{\p_n v^k}{\p_n v^N} &\\
 \qquad+ \sum_{i=1}^{n-1}\p_i v^N (\frac{\p_i v^N}{\p_n v^N}\p_n v^k - \p_i v^k)] = \l_k (\p_n v^N) v^k & k=1,\ldots,N-1  \text{ on } Q_r^+ \sm \p H\\
  -\sum_{i=1}^{n-1}\p_i\p_i v^N + \p_n [\frac{1}{\p_n v^N} + \sum_{i=1}^{n-1}\frac{(\p_i v^N)^2}{\p_n v^N}] = \l_N (\p_n v^N) y_n &  \text{ on } Q_r^+\sm \p H\\
  v^k = 0 & k = 1,\ldots, N-1  \text{ on } Q_r^+ \cap H \\
  (1 + \sum_{i=1}^{n-1}(\p_i v^N)^2)(1 + \sum_{k=1}^{N-1} \p_n v^k) = \p_n v^N &  \text{ on } Q_r^+ \cap \p H.
\end{cases}
\] 
We know that the $v^k$ are in $C^{1,\a}(Q_r^+)\cap C^\8_{\text{loc}}(Q_r^+)$, and have $\p_i v^k(0) = 0$ for $i<n$, but $\p_n v^N(0)>0$.

This system is of the form demanded by Lemma \ref{lem:schauder3}, after possibly making $r$ smaller: the coefficients $A^l(\n v)$ are triangular, analytic in the parameters, and when their derivatives are evaluated at $\n v(0)$  they satisfy assumptions (B),(C). These assumptions then remain satisfied on an open set around $0$.

We apply (a suitably scaled version of) Lemma \ref{lem:schauder3}, to obtain that $v^k\in C^{2,\a}(Q_{r/3}^+)$. Then we may proceed as in \cite{KNS}, applying \cite[Theorem 6.8.2]{Morrey} to conclude that the $v^k$ are analytic up to $\p H$; the analyticity of $v^N$ in particular implies that $\p \W$ is analytic on a neighborhood of $0$. Repeating the argument at each point of $\p^*\W$ implies the conclusion.
\end{proof}

\begin{remark}
	We believe the Schauder theorem in Corollary \ref{cor:schauder2} to be a rather standard result, even though there does not appear to be any convenient reference in the literature. In the scalar case, it may be found in the book \cite{L}, along with analogues of Lemma \ref{lem:schauder3}. The main point in the vectorial setting is that although the a priori regularity is lower than what is used in \cite{ADN2}, the representation formulas and uniqueness results from there may still be applied, after suitable approximation arguments. While we use the fact that the system is triangular to simplify the proof in several places (and to simplify verifying ellipticity and the coercivity of the boundary condition), this should not be regarded as essential. We did not use the particular structure of the boundary condition: any first-order set of boundary conditions (for a purely second-order system) which satisfy the complementing condition would work, with suitably modified statements.
\end{remark}

\section*{Acknowledgments} DK was supported by the NSF MSPRF fellowship DMS-1502852. FL was supported by the NSF grant DMS-1501000.

\bibliographystyle{plain}
\bibliography{nondegenerate}

\end{document}